\documentclass[a4paper]{article}
\usepackage[utf8x]{inputenc}
\usepackage{a4wide}
\usepackage{amssymb}
\usepackage{amsmath}
\usepackage{amsthm}
\usepackage{dsfont}
\usepackage{graphicx}
\usepackage{pdflscape}
\usepackage{rotating}
\usepackage{csquotes}
\usepackage{subcaption}

\newcommand{\etal}{\textit{et al.}\ }
\newcommand{\R}{\mathbb{R}}

\newcommand{\diag}{\operatorname{diag}}

\renewcommand{\vec}[1]{\mathbf{#1}}
\newcommand{\mat}[1]{\vec{\vec{#1}}}

\newcommand{\Id}{\mat{I}}
\newcommand{\pdfs}{particle distribution functions}

\newcommand{\zero}{\mat{\emptyset}}
\renewcommand{\R}{\mathbb{R}}
\renewcommand{\O}{\mathcal{O}}

\DeclareMathOperator*{\Span}{span}

\DeclareMathOperator*{\Dim}{dim}

\newtheorem{theorem}{Theorem}

\newtheorem{remark}[theorem]{Remark}
\newtheorem{definition}[theorem]{Definition}

\author{Philipp Otte\footnote{Forschungszentrum J\"ulich GmbH, Institute for Advanced Simulation, J\"ulich Supercomputing Centre, 52425 J\"ulich, Germany} \footnote{MathCCES, RWTH Aachen University, Schinkelstra\ss e 2,\ 52056 Aachen, Germany} \footnote{Corresponding author: p.otte@fz-juelich.de} \footnote{ORCiD ID: 0000-0002-1586-2274} and Martin Frank\footnote{Karlsruhe Institute of Technology, Steinbuch Center for Computing, Hermann-von-Helmholtz-Platz 1, 76344 Eggenstein-Leopoldshafen, Germany}}
\title{A Structured Approach to the Construction of Stable Linear Lattice Boltzmann Collision Operator}
\begin{document}
\maketitle

\begin{abstract}
We introduce a structured approach to the construction of linear BGK-type collision operators ensuring that the resulting Lattice-Boltzmann methods are stable with respect to a weighted $L^2$-norm.
The results hold for particular boundary conditions including periodic, bounce-back, and bounce-back with flipping of sign boundary conditions.
This construction uses the equivalent moment-space definition of BGK-type collision operators and the notion of stability structures as guiding principle for the choice of the equilibrium moments for those moments influencing the error term only but not the order of consistency.
The presented structured approach is then applied to the 3D isothermal linearized Euler equations with non-vanishing background velocity.
Finally, convergence results in the strong discrete $L^\infty$-norm highlight the suitability of the structured approach introduced in this manuscript.
\end{abstract}

\section{Introduction}
\label{sec: Introduction}

The Lattice-Boltzmann method (LBM) has proven to be a viable tool for simulating flows governed by the Navier-Stokes equations.
The LBM offers two major benefits: first, simplicity of implementation; and second, its fully explicit structure.
Due to this explicit nature and the beneficial communication pattern, it suites today's distributed memory high-performance computer architectures extremely well.
In addition, its stream-collide structure also maps well on GPUs (see Geier and Schönherr \cite{GeierSchoenherr2017}).
Still, the LBM should not only be seen as a means for simulating flows governed by the Navier-Stokes equations but also allows the solution of different systems of partial differential equations (see for example Kataoka and Tsutahara \cite{kataoka2004lattice} and Chai \etal \cite{ChaiHeGuoShi2018}).
\medskip
\\
In the derivation of numerical approximations, two properties are paramount: consistency and stability.
While consistency is usually known \textit{a priori}, stability often can only be assessed \textit{a posteriori}.
Hence it is beneficial to include some notion simplifying the assessment of stability into the process of deriving a numerical approximation.
This is particularly important for the LBM which allows for many degrees of freedom in the derivation of collision operators since usually not for all moments which are supported by the used velocity set consistency prescribes an associated equilibrium moment.
In the most popular collision operator (compare He and Luo \cite{HeLuo97From}), these equilibrium moments not prescribed by consistency are chosen such that the equilibrium distribution is an expansion of the Maxwell distribution.
A different approach is chosen in the field of the entropic LBM (see for example Bogosian \etal \cite{EntropicLatticeBoltzmannMethods} and Chikatamarla \etal \cite{EntropicLBM3D}) where the equilibrium moments are chosen such that the equilibrium distribution minimizes an entropy function.
A third approach (see for example Dubois and Lallemand \cite{DuboisLallemand2009} \cite{DuboisLallemand2011}) chooses the equilibrium moments such that the overall consistency error is minimized.
Out of these three approaches, only the entropic one has a notion of stability built into its process of derivation.
Still, in the entropic LBM, strict second-order consistency at every point in time and lattice node is sacrificed to ensure stability.
\medskip
\\
In Otte and Frank \cite{Otte2016}, the authors describe the derivation of stable and second-order consistent LBMs for the compressible Euler equations with zero background velocity for mono- and polyatomic gases.
This derivation is based on the discretization of the equilibrium in the acoustic limit presented by Bardos \etal \cite{Bardos2000} and used an \textit{a posteriori} stability analysis using sufficient conditions for stability in the discrete $L^2$-norm also derived in Otte and Frank \cite{Otte2016}.
The attempts of one of the authors to derive stable LBMs for the isothermal linearized Euler equations with non-vanishing background velocity using the afore mentioned approaches proved not successful.
The main challenges in this derivation are that approaches based on expansion of the Maxwellian distribution require an excessive number of moments to be supported by the velocity set or that stability of the method cannot be analyzed in general but only on a case to case basis.
\medskip
\\
The notion of stability structures (introduced by Banda \etal \cite{banda2006}, extended analysis and application by Junk and Yong \cite{JunkYong2009}, Rheinländer \cite{Rheinlaender2010}, and Yong \cite{Yong2009862}) presents a structured approach towards the assessment of stability of LBMs in weighted $L^2$-norms. This notion is based on the stability theory for hyperbolic relaxation systems by Yong (cf. \cite{Yong2001}).
In this work, we use the notion of stability structures as guiding principle for the derivation of stable LBMs using linear collision operators.
We first show that a clever choice of the moment matrix used to map particle densities onto moments drastically lowers the number of conditions for the existence of a stability structure.
Then, the number of conditions is reduced further by altering the moment matrix such that a subset of the condition is automatically fulfilled.
\medskip
\\
This manuscript is structured as follows.
In section \ref{sec: LBM} the LBM with BGK-type collision operators and the notion of stability structures are introduced.
Subsequently, the conditions for existence of stability structures are applied to  linear BGK-type collision operators with collisions defined in moment space in section \ref{sec: BGK}.
In the following section \ref{sec: relative}, we introduce the concept of relative collision operators and show that the formulation of BGK-type collision operators using relative moment matrices lowers the number of conditions for existence of a stability structure.
This concept is extended to partially relative collision operator which eliminates a subset of the conditions for existence of a stability structure by applying a truncated Gram-Schmid orthogonalization to a subset of the rows of the relative moment matrix.
In section \ref{sec: 3D LEE}, we use the concept of partially relative schemes to construct collision operators for the 3D isothermal linearized Euler equations.
Finally, we present convergence results for the 3D isotherma linearized Euler equations in section \ref{sec: results} and conclude with some further remarks.

\section{The Lattice-Boltzmann Method}
\label{sec: LBM}

In the LBM, a system of partial differential equations is not discretized directly as is the case in most approaches including finite differences, finite volume methods, the finite element method, or spectral methods.
Instead, the evolution of the densities of virtual particles is simulated on a lattice.
In this work, the lattice $\mathcal{L} \subset \R^D$ is assumed periodic and equidistant with dimensionless spacing $\Delta t$ ($D$ denotes the  spatial dimension).
The simulation advances the densities on the temporal grid $\mathcal{T} = \left\{ k \Delta t: k = 0,\dots,k_{\text{final}}\right\}$ with the same dimensional spacing $\Delta t$.
\medskip
\\
For each lattice node $\vec{x} \in \mathcal{L}$ and timestep $t \in \mathcal{T}$, we assume $n$ densities $f_i(t,\vec{x})$ for $i=1,\dots,n$ for virtual particles moving with velocities $\vec{c}_i \in \R^D$.
The densities at node $\vec{x}$ and timestep $t$ are collected in the density vector $\vec{f}(t,\vec{x}) = ( f_1(t,\vec{x}),\dots,f_n(t,\vec{x}))^T$.
The discrete velocities $\vec{c}_i$ for $i=1,\dots,n$ are chosen such that:
\begin{itemize}
	\item if $\vec{x} \in \mathcal{L}$ then $\vec{x} + \Delta t \vec{c}_i \in \mathcal{L}$; and
	\item the velocity set $\mathcal{S} = \left\{ \vec{c}_i: i=1,\dots,n \right\}$ is symmetric, i.e. for $\vec{c} \in \mathcal{S}$ also $-\vec{c} \in \mathcal{S}$.
\end{itemize}
In this manuscript, the particle densities are evolved according to the Lattice-Boltzmann equation with BGK-type collision operator:
\begin{align}
	f_i(t+\Delta t,\vec{x} + \Delta t \vec{c}_i) = f_i(t,\vec{x}) + \frac{1}{\tau} \left( f_i^{eq}[\vec{f}(t,\vec{x})] - f_i(t,
\vec{x}) \right).\label{eq: LBE}
\end{align}
Here, $f_i^{eq}: \R^n \to \R$ for $i=1,\dots,n$ denotes the mapping of the particle densities onto their associated equilibrium particle densities and $\tau \in \R^+$ the relaxation time with which the particle densities are relaxed towards their equilibrium.
Macroscopic quantities are obtained from the particle densities via computation of moments, i.e. weighted sums:
\begin{align}
	\sum_{i=1}^n \zeta(\vec{c}_i) f_i(t,\vec{x}),
\end{align}
with different weight functions $\zeta(\cdot)$.
Important moments for LBMs simulating fluid dynamics are density and momentum:
\begin{align}
	\sum_{i=1}^n f_i(t,\vec{x}) = \rho(t,\vec{x}), \quad \sum_{i=1}^n \vec{c}_i f_i(t,\vec{x}) = \rho \vec{u}.
\end{align}
A moment defined by function $\zeta(\cdot)$ is called conserved moments if:
\begin{align*}
\sum_{i=1}^n \zeta(\vec{c}_i) f_i(t,\vec{x}) = \sum_{i=1}^n \zeta(\vec{c}_i) f^{eq}_i[\vec{f}(t,\vec{x})]\ \text{for every}\ \vec{f} \in \R^D.
\end{align*}
The macroscopic system of partial differential equations solved by an LBM is obtained using a combination of Taylor expansion of (\ref{eq: LBE}), asymptotic expansion (Caiazzo \etal \cite{AnalysisTechniques2009}), Chapman-Enskog expansion (Junk \etal \cite{Junk2005}), or Maxwell iteration (see Yong \etal \cite{YongZhaoLuo21016}), and mapping of moments.
Note that the choice of the maps $f_i^{eq}$, the velocity set $\mathcal{S}$, and the relaxation time $\tau$ determine which macroscopic system of partial differential equations is solved.
\medskip
\\
Based on Yong's \cite{Yong2001} results for stability of hyperbolic relaxation systems, Banda, Yong, and Klar \cite{banda2006} and Junk and Yong \cite{JunkYong2009} introduced the notion of stability structures providing a structured approach to stability analysis for LBMs.
In this manuscript, only linear collision operators are considered such that collision term $\frac{1}{\tau} \left( f_i^{eq}[\vec{f}(t,\vec{x})] - f_i(t,\vec{x}) \right)$ can be written in linear form:
\begin{align}
	\label{eq: linear collision term}
	\mat{J} = \frac{1}{\tau}\left( \mat{E}_{pdf} - \Id \right),
\end{align}
with $\mat{E}_{pdf}\ \vec{f} = \vec{f}^{eq}[\vec{f}]$ for all $\vec{f} \in \R^n$.
In the remainder of this work, $\mat{J}$ is called collision matrix.
With this, we can now introduce the notion of (pre-)stability structures for collision matrices $\mat{J}$ of form (\ref{eq: linear collision term}).
\begin{definition}{Rheinländer \cite{Rheinlaender2010}}
\label{def: stab struct}
A collision matrix $\mat{J} \in \R^{n\times n}$ is said to have a pre-stability structure $\left(\mat{P},\vec{a},\vec{s}\right)$ if there exists an invertible matrix $\mat{P} \in \R^{n \times n}$ and vectors $\vec{s} = -\diag(s_1,\dots,s_n)^T \in \R^n$ and $\vec{a} = (a_1,\dots,a_n)^T \in \R^n$ such that
\begin{enumerate}
	\item $\mat{P}\ \mat{J} = - \diag(s_1,\dots,s_n)\ \mat{P}$,\\
	\item $\mat{P}^T\ \mat{P} = \diag(a_1,\dots,a_n)$.
\end{enumerate}
Moreover, the pre-stability structure becomes a stability structure if
\begin{align*}
	s_k \in [0,2]\ \text{for all}\ k \in \{1,\dots,n\}.
\end{align*}
\end{definition}
Based on the concept of stability structures, Junk and Yong \cite{JunkYong2009} stated the following Theorem \ref{th: stab struct}:
\begin{theorem}{Junk and Yong \cite{JunkYong2009}}
\label{th: stab struct}
Assume the collision matrix of a LBM admits a stability structure $\left(\mat{P},\vec{a},\vec{s}\right)$ in the sense of Definition \ref{def: stab struct}.
Then, for periodic initial data and periodic lattice $\mathcal{L}$ solution $\vec{f}(\cdot,\cdot)$ is stable in a weighted $L^2$-norm induced by $\mat{P}$:
\begin{align}
	\label{eq: stab weighted L2}
	||\vec{f}(t+\Delta t, \cdot)||_{\mat{P}} \leq ||\vec{f}(t,\cdot)||_{\mat{P}},
\end{align}
with
\begin{align*}
	||\vec{f}(t,\cdot)||_{\mat{P}} = \sum_{\vec{x} \in \mathcal{L}} ||\vec{f}(t,\vec{x})||_{\mat{P}}
\end{align*}
and weighted $L^2$-norm induced by $\mat{P}$:
\begin{align*}
	||\vec{f}(t,\vec{x})||_{\mat{P}} = \left( ||\mat{P}\ \vec{f}(t,\vec{x})||^2 \right)^{\frac{1}{2}}.
\end{align*}
\end{theorem}
\begin{remark}{}
Note that the same result holds true for bounce-back and bounce-back with flipping of sign boundary conditions (cf. Junk and Yong \cite{JunkYong2009} and Rheinländer \cite{Rheinlaender2010}). 
\end{remark}
The algorithmic construction of stable collision operators is based on the following Theorem \ref{th: symmetry condition} derived from Theorem \ref{th: stab struct} (see Rheinländer \cite{Rheinlaender2010} and Yong \cite{Yong2009862}).
\begin{theorem}{Rheinländer \cite{Rheinlaender2010}}
\label{th: symmetry condition}
Collision matrix $\mat{J} \in \R^{n \times n}$ admits a pre-stability structure in the sense of Definition \ref{def: stab struct} if and only if there exists a diagonal and positive definite matrix $\mat{\Lambda} \in \R^{n\times n}$ such that:
\begin{align}
	\label{eq: contion existence pre-stability structure}
	\mat{J}\ \mat{\Lambda} = \mat{\Lambda}\ \mat{J}^T.
\end{align}
\end{theorem}

\section{Classical Linear Collision Operators}
\label{sec: BGK}

Assume a linear BGK operator that relaxes all moments to the according equilibrium moments with the same relaxation time.
Assume a regular matrix $\mat{M} \in \R^{n \times n}$ which maps the \pdfs\ onto the $n$-dimensional moment space.
We denote matrix $\mat{M}$ as \enquote{moment matrix}.
Mapping the equilibrium distribution function $\vec{f}^{eq}[\vec{f}]$ onto moment space, one obtains the equilibrium moment vector $\vec{m}^{eq}(\vec{f}) = \mat{M}\ \vec{f}^{eq}$.
Due to the linearity of the map $\vec{f} \mapsto \vec{f}^{eq}[\vec{f}]$, the mapping from moments onto the equilibrium moment vector $\vec{m}^{eq}[\vec{f}]$ is linear as well and can be written using a suitable matrix $\mat{E} \in \R^{n \times n}$:
\begin{align*}
	\vec{m}^{eq}[\vec{f}] = \mat{E}\ \vec{m} = \mat{E}\ \mat{M}\ \vec{f}.
\end{align*}
Note that while matrix $\mat{E}_{pdf}$, introduced above, maps the particle densities onto the equilibrium densities, the matrix $\mat{E}$ maps moments onto the according equilibrium moments.
Hence, the equilibrium distribution $\vec{f}^{eq}[\vec{f}]$ can be computed as:
\begin{align*}
	\vec{f}^{eq}[\vec{f}] = \underbrace{\mat{M}^{-1}\ \mat{E}\ \mat{M}}_{=: \mat{H}}\ \vec{f}.
\end{align*}
\medskip
\\
For the consistency analysis of the LBM, additional moments to the conserved moments of the equilibrium distribution $\vec{f}^{eq}[\vec{f}]$ are required.
Usually, these moments are of the form $\sum\limits_{i=1}^N \prod\limits_{j=1}^D c_{ij}^{k_j} f_i$ with exponents $k_j$ for $j=1..D$.
Assume, a LBM has $\gamma$ conserved moments and prescribes $\beta$ moments of the equilibrium distribution.
Let $\mat{M}^{(1)}$ denote a regular matrix with the following properties:
\begin{enumerate}
	\item the first $\gamma$ rows map the particle distribution vector $\vec{f}$ onto the conserved moments;
	\item the following $\beta$ rows map the particle distribution vector $\vec{f}$ onto the moments for which equilibrium moments need to be prescribed for consistency; and
	\item the remaining $N-\beta-\gamma$ rows are chosen such that $\mat{M}$ is regular.
\end{enumerate}
Using these assumptions let  the collision operator $\mat{J}^{(1)}$ be defined as:
\begin{align}
\mat{J}^{(1)} &= \frac{1}{\tau} \left( \mat{H} - \Id_n \right)\ \vec{f}\nonumber\\
&= \frac{1}{\tau} \mat{M}^{(1)-1}\ \left( \mat{E}^{(1)} - \Id_n \right)\ \mat{M}^{(1)}. \label{eq: collision linear}
\end{align}
The matrix $\mat{J}^{(1)}$ admits a pre-stability structure if and only if a diagonal positive definite matrix $\mat{\Lambda}^{(1)}$ exists such that:
\begin{align}
\mat{J}^{(1)}\ \mat{\Lambda}^{(1)} &\overset{!}{=} \mat{\Lambda}^{(1)}\ \mat{J}^{(1)T}.\label{eq: unmodified symmetry condition}
\end{align}
Plugging in (\ref{eq: collision linear}) into (\ref{eq: unmodified symmetry condition}), one obtains the condition:
\begin{align*}
\left( \mat{M}^{(1)-1}\ \mat{E}^{(1)}\ \mat{M}^{(1)} - \Id_n \right)\ \mat{\Lambda}^{(1)} &\overset{!}{=} \mat{\Lambda}^{(1)} \left( \mat{M}^{(1)T}\ \mat{E}^{(1)T}\ \mat{M}^{(1)-T} - \Id_n \right),
\end{align*}
which reduces to:
\begin{align}
\mat{M}^{(1)-1}\ \mat{E}^{(1)}\ \mat{M}^{(1)}\ \mat{\Lambda}^{(1)} &\overset{!}{=} \mat{\Lambda}^{(1)}\ \mat{M}^{(1)T}\ \mat{E}^{(1)T}\ \mat{M}^{(1)-T}. \label{eq: stab struct 1}
\end{align}
Since the moment matrix $\mat{M}^{(1)}$ is assumed regular, equation (\ref{eq: stab struct 1}) can be replaced by the equivalent condition:
\begin{align}
\mat{E}^{(1)}\ \mat{M}^{(1)}\ \mat{\Lambda}^{(1)}\ \mat{M}^{(1)T} &\overset{!}{=} \mat{M}^{(1)}\ \mat{\Lambda}^{(1)}\ \mat{M}^{(1)T}\ \mat{E}^{(1)T}.\label{eq: stab struct 2}
\end{align}
Introducing the short-cut notation $\tilde{\mat{\Lambda}}^{(1)} = \mat{M}^{(1)}\ \mat{\Lambda}^{(1)}\ \mat{M}^{(1)T}$, condition (\ref{eq: stab struct 2}) reads:
\begin{align}
\mat{E}^{(1)}\ \tilde{\mat{\Lambda}}^{(1)} &\overset{!}{=} \tilde{\mat{\Lambda}}^{(1)}\ \mat{E}^{(1)T}. \label{eq: condition general}
\end{align}
Therefore, condition (\ref{eq: unmodified symmetry condition}) on the collision operator $\mat{J}^{(1)}$ has been recast as a condition on the  equilibrium matrix $\mat{E}^{(1)}$ which needs to be symmetrized by the matrix:
\begin{align}
\label{eq: def lambda(1)}
\tilde{\Lambda}^{(1)} = \Big( < \vec{r}^{(1)}_i, \vec{r}^{(1)}_j >_{\mat{\Lambda}^{(1)}} \Big)_{i,j=1,\dots,n},
\end{align}
where the vector $\vec{r}^{(1)}_i$ denotes the $i$-th row of $\mat{M}^{(1)}$ and the scalar product $<\cdot,\cdot>_{\mat{\Lambda}^{(1)}}$ is defined using matrix $\mat{\Lambda}^{(1)}$:
\begin{align}
\label{eq: scalar product}
<\vec{\xi},\vec{\zeta}>_{\mat{\Lambda}^{(1)}} &= \vec{\xi}\ \mat{\Lambda}^{(1)}\ \vec{\zeta}^T.
\end{align}
Since we assume that all equilibrium moments are linear combinations of the conserved moments, the equilibrium matrix $\mat{E}$ is of the form:
\begin{align*}
\mat{E}^{(1)} &= \begin{pmatrix}
\Id_\gamma & \zero_{\gamma \times \beta} & \zero_{\gamma \times (n-\beta-\gamma)}\\
\mat{E}^{(1)}_{21} & \zero_{\beta \times \beta} & \zero_{\beta \times (n-\beta-\gamma)}\\
\mat{E}^{(1)}_{31} & \zero_{(n-\beta-\gamma) \times \beta} & \zero_{(n-\beta-\gamma) \times (n-\beta-\gamma)}
\end{pmatrix},
\end{align*}
where $\mat{E}^{(1)}_{21}$ and $\mat{E}^{(1)}_{31}$ map the conserved moments onto the $\beta$ equilibrium moments needed for consistency analysis and the remaining $n-\beta-\gamma$ equilibrium moments not needed for consistency analysis, respectively.
\medskip
\\
For the further analysis, we write matrix $\tilde{\mat{\Lambda}}^{(1)}$ as block matrix:
\begin{align}
\label{eq: block structure lambda1}
\tilde{\Lambda}^{(1)} &= \begin{pmatrix}
\tilde{\Lambda}^{(1)}_{11} & \tilde{\Lambda}^{(1)}_{12} & \tilde{\Lambda}^{(1)}_{13}\\
\tilde{\Lambda}^{(1)}_{21} & \tilde{\Lambda}^{(1)}_{22} & \tilde{\Lambda}^{(1)}_{23}\\
\tilde{\Lambda}^{(1)}_{31} & \tilde{\Lambda}^{(1)}_{32} & \tilde{\Lambda}^{(1)}_{33}
\end{pmatrix},
\end{align}
with $\tilde{\mat{\Lambda}}^{(1)}_{ij} \in \R^{d_i \times d_j}$ with $d_1 = \gamma$, $d_2 = \beta$, $d_3 = n-\beta-\gamma$.
In matrix $\mat{\Lambda}^{(1)}$, each block $\tilde{\mat{\Lambda}}^{(1)}_{ij}$ contains information on the $\mat{\Lambda}$-scalar products of the rows for either the conserved moments, momenst necessary for consistency, or remaining moments with the rows for either the conserved moments, momenst necessary for consistency, or remaining moments.
With this, condition (\ref{eq: condition general}) reads:
\begin{align}
\label{eq: condition general plugged in}
\begin{pmatrix}
\tilde{\mat{\Lambda}}^{(1)}_{11} & \tilde{\mat{\Lambda}}^{(1)}_{12} & \tilde{\mat{\Lambda}}^{(1)}_{13}\\ 
\mat{E}^{(1)}_{21}\ \tilde{\mat{\Lambda}}^{(1)}_{11} & \mat{E}^{(1)}_{21}\ \tilde{\mat{\Lambda}}^{(1)}_{12} & \mat{E}^{(1)}_{21}\ \tilde{\mat{\Lambda}}^{(1)}_{13}\\
\mat{E}^{(1)}_{31}\ \tilde{\mat{\Lambda}}^{(1)}_{11} & \mat{E}^{(1)}_{31}\ \tilde{\mat{\Lambda}}^{(1)}_{12} & \mat{E}^{(1)}_{31}\ \tilde{\mat{\Lambda}}^{(1)}_{13}
\end{pmatrix} \overset{!}{=}
\begin{pmatrix}
\tilde{\mat{\Lambda}}^{(1)}_{11} & \tilde{\mat{\Lambda}}^{(1)}_{11} \mat{E}^{(1)T}_{21} & \tilde{\mat{\Lambda}}^{(1)}_{11} \mat{E}^{(1)T}_{31}\\
\tilde{\mat{\Lambda}}^{(1)}_{12} & \tilde{\mat{\Lambda}}^{(1)}_{12} \mat{E}^{(1)T}_{21} & \tilde{\mat{\Lambda}}^{(1)}_{12} \mat{E}^{(1)T}_{31}\\
\tilde{\mat{\Lambda}}^{(1)}_{13} & \tilde{\mat{\Lambda}}^{(1)}_{13} \mat{E}^{(1)T}_{21} & \tilde{\mat{\Lambda}}^{(1)}_{13} \mat{E}^{(1)T}_{31}
\end{pmatrix}.
\end{align}
\begin{remark}[]
Constructing a combination of $\mat{E}^{(1)}$, $\mat{M}^{(1)}$, and $\mat{\Lambda}^{(1)}$ satisfying condition (\ref{eq: condition general plugged in}) is, in general, rather complicated.
Thus, we require a structured approach towards the construction of these matrices.
\end{remark}

\section{Relative Schemes}
\label{sec: relative}

From the general structure of condition (\ref{eq: condition general plugged in}) and the inherent symmetry of $\tilde{\mat{\Lambda}}_{11}^{(1)}$, one can easily see that four lower-dimensional matrices are of importance: $\tilde{\mat{\Lambda}}^{(1)}_{12}$, $\tilde{\mat{\Lambda}}^{(1)}_{13}$, $\mat{E}^{(1)}_{21}$, and $\mat{E}^{(1)}_{31}$.
We now simplify condition (\ref{eq: condition general plugged in}) by eliminating all terms depending on matrices $\mat{E}^{(1)}_{21}$ and $\mat{E}^{(1)}_{31}$.
We achieve this by introducing a modified moment matrix $\mat{M}^{(2)}$.
Since all non-conserved equilibrium moments are linear combinations of the conserved moments, it is possible to construct new moment matrix $\mat{M}^{(2)}$ with rows $\vec{r}^{(2)}_i$ for $i=1+\gamma,..,n$ chosen in such a way that all according non-conserved equilibrium moments vanish:
\begin{align}
\label{eq: fully relative moment matrix}
\mat{M}^{(2)} = \begin{pmatrix}
\Id_\gamma & \zero_{\gamma \times \beta} & \zero_{\gamma \times (n-\beta-\gamma)}\\
-\mat{E}^{(1)}_{21} & \Id_\beta & \zero_{\beta \times (n-\beta-\gamma)}\\
-\mat{E}^{(1)}_{31} & \zero_{(n-\beta-\gamma) \times \beta} & \Id_{(n-\beta-\gamma)}
\end{pmatrix}\ \mat{M}^{(1)}.
\end{align}
The corresponding matrix mapping the moments onto the equilibrium moments $\mat{E}^{(2)}$ is given by:
\begin{align}
\label{eq: fully relative equilibrium matrix}
\mat{E}^{(2)} = \begin{pmatrix}
\Id_\gamma & \zero_{\gamma \times \beta} & \zero_{\gamma \times (n-\beta-\gamma)}\\
\zero_{\beta \times \gamma} & \zero_{\beta \times \beta} & \zero_{\beta \times (n-\beta-\gamma)}\\
\zero_{(n-\beta-\gamma) \times \gamma} & \zero_{(n-\beta-\gamma) \times \beta} & \zero_{(n-\beta-\gamma) \times (n-\beta-\gamma)}
\end{pmatrix}.
\end{align}
With this, analog of condition (\ref{eq: condition general plugged in}) for $\mat{M}^{(2)}$, $\mat{E}^{(2)}$, and $\mat{\Lambda}^{(2)}$ reads:
\begin{align}
\label{eq: condition general plugged in fully relative}
\begin{pmatrix}
\tilde{\mat{\Lambda}}^{(2)}_{11} & \tilde{\mat{\Lambda}}^{(2)}_{12} & \tilde{\mat{\Lambda}}^{(2)}_{13}\\ 
\zero_{\beta \times \gamma} & \zero_{\beta \times \beta} & \zero_{\beta \times (n-\beta-\gamma)}\\
\zero_{(n-\beta-\gamma) \times \gamma} & \zero_{(n-\beta-\gamma) \times \beta} & \zero_{(n-\beta-\gamma) \times (n-\beta-\gamma)}
\end{pmatrix} \overset{!}{=}
\begin{pmatrix}
\tilde{\mat{\Lambda}}^{(1)}_{11} & \zero_{\gamma\times \beta} & \zero_{\gamma \times (n-\beta-\gamma)}\\
\tilde{\mat{\Lambda}}^{(1)}_{12} & \zero_{\beta \times \beta} & \zero_{\beta \times (n-\beta-\gamma)}\\
\tilde{\mat{\Lambda}}^{(1)}_{13} & \zero_{(n-\beta-\gamma)\times \beta} & \zero_{(n-\beta-\gamma) \times (n-\beta-\gamma)}
\end{pmatrix},
\end{align}
where $\tilde{\mat{\Lambda}}^{(2)}_{11}$ for $j=1,..,3$ is defined by:
\begin{align*}
\tilde{\mat{\Lambda}}^{(2)}_{1j} &= \sum\limits_{k=1}^3 \mat{M}^{(2)}_{1k}\ \mat{\Lambda}^{(2)}\ \mat{M}^{(2)T}_{jk}.
\end{align*}
This definition directly corresponds to definition (\ref{eq: def lambda(1)}) and block structure (\ref{eq: block structure lambda1}).
Hence, we find that for in order to satisfy condition (\ref{eq: condition general plugged in fully relative}) the entries of $\tilde{\mat{\Lambda}}^{(2)}_{12}$ and $\tilde{\mat{\Lambda}}^{(2)}_{13}$ need to vanish.
This is equivalent to the following orthogonality condition of the rows of $\mat{M}^{(2)}$ w.r.t. scalar product (\ref{eq: scalar product}) for $\mat{\Lambda}^{(2)}$:
\begin{align}
\label{eq: cond scalar product fully relative}
< \vec{r}^{(2)}_i, \vec{r}^{(2)}_j>_{\mat{\Lambda}^{(2)}} &\overset{!}{=} 0\ \text{for}\ i=1,..,\gamma,\ j=1+\gamma,..,n.
\end{align}
Condition (\ref{eq: cond scalar product fully relative}) can also be written as condition on the spaces spanned by the rows corresponding to conserved moments, moments with prescribed equilibrium moments, and the remaining moments:
\begin{align}
\Span\left\{ \vec{r}_i^{(2)}: i = 1+\beta,\dots,\beta+\gamma \right\} &\overset{!}{\subset} \Span\left\{ \vec{r}_i^{(2)}: i = 1,\dots,\gamma \right\}^{\perp_{\mat{\Lambda}^{(2)}}},\label{eq: span prescribed}\\
\Span\left\{ \vec{r}_i^{(2)}: i = 1+\beta+\gamma,\dots,n \right\} &\overset{!}{\subset} \Span\left\{ \vec{r}_i^{(2)}: i = 1,\dots,\gamma \right\}^{\perp_{\mat{\Lambda}^{(2)}}},\label{eq: span remaining}
\end{align}
where the superscript $\perp_{_{\mat{\Lambda}^{(2)}}}$ denotes the orthogonal complement of the set w.r.t. the scalar product $<\cdot,\cdot>_{\mat{\Lambda}^{(2)}}$. 
Theorem \ref{th: Fully Relative} summarizes these results:
\begin{theorem}[]
\label{th: Fully Relative}
Assume a linear LBM BGK scheme 
\begin{align*}
	f_i(t+\Delta t,\vec{x} + \Delta t \vec{c}_i) = f_i(t,\vec{x}) + \frac{1}{\tau} \vec{e}_i^T\ \left( \mat{E}^{(1)} - \Id \right)\ \vec{f}(t,\vec{x}),
\end{align*}
with moment matrix $\mat{M}^{(1)}$ and equilibrium matrix $\mat{E}^{(1)}$ which is consistent to $\alpha$-th order.
In addition, assume the relative moment matrix $\mat{M}^{(2)}$ defined according to (\ref{eq: fully relative moment matrix}), the equilibrium matrix $\mat{E}^{(2)}$ defined according to (\ref{eq: fully relative equilibrium matrix}), and a diagonal and positive definite matrix $\mat{\Lambda}^{(2)}$ defined analog to (\ref{eq: block structure lambda1}) fulfilling condition (\ref{eq: cond scalar product fully relative}).
Then, the collision matrix 
\begin{align*}
\mat{J}^{(2)} = \frac{1}{\tau} \mat{M}^{(2)-1}\ \left( \mat{E}^{(2)} - \Id_n \right)\ \mat{M}^{(2)}
\end{align*}
admits a pre-stability structure.
If in addition, $\tau \in \Big[ \frac{1}{2}, \infty \Big)$, then the collision matrix $\mat{J}^{(2)}$ admits a stability structure.
\end{theorem}
\begin{proof}
With $\mat{M}^{(2)}$ and $\mat{\Lambda}^{(2)}$ fulfilling condition (\ref{eq: cond scalar product fully relative}), $\tilde{\mat{\Lambda}}^{(2)}_{12} = \zero_{\gamma \times \beta}$ and $\tilde{\mat{\Lambda}}^{(2)}_{13} = \zero_{\gamma \times (n-\beta-\gamma)}$ follow directly.
With $\mat{E}^{(2)}$ fulfilling condition (\ref{eq: fully relative equilibrium matrix}), we find:
\begin{align*}
\mat{E}^{(2)}\ \tilde{\mat{\Lambda}}^{(2)} &= \tilde{\mat{\Lambda}}^{(2)}\ \mat{E}^{(2)}
\end{align*}
After multiplication with $\mat{M}^{(2)-1}$ from left and $\mat{M}^{(2)-T}$ from right, one finds that
\begin{align*}
	\mat{\Lambda}^{(2)} = \mat{M}^{(2)-1}\ \mat{\Lambda}^{(2)}\ \mat{M}^{(2)-T}
\end{align*}
symmetrizes the collision matrix $\mat{J}^{(2)}$:
\begin{align*}
\mat{J}^{(2)}\ \mat{\Lambda}^{(2)} &= \mat{\Lambda}^{(2)}\ \mat{J}^{(2)T}.
\end{align*}
Hence, $\mat{J}^{(2)}$ admits a pre-stability structure.
\medskip
\\
For proof of existence of a stability structure for $\tau \in \Big[ \frac{1}{2}, \infty \Big)$, the structure of the collision operator is used.
First, we observe that $\mat{M}^{(2)-1}\ \mat{E}^{(2)}\ \mat{M}^{(2)}$ is a projection matrix.
Hence, the eigenvalues of collision operator $\mat{J}^{(2)}$ are $-\frac{1}{\tau}$ and $0$.
With $\tau \geq \frac{1}{2}$, we find: $\sigma\left( \mat{J}^{(2)} \right) \subset [-2,0]$.
\end{proof}
\begin{remark}[]
\label{re: Relative Schemes}
Theorem \ref{th: Fully Relative} presents a stability result for a class of BGK-type collision operators which can be rewritten in a form, such that the matrix mapping conserved moments onto the equilibrium moments is of the form defined in (\ref{eq: fully relative equilibrium matrix}).
In this work, such schemes are called fully relative schemes where the term relative is meant to resemble the term relative velocity scheme used by Dubois \etal \cite{RelativeVelocity}.
The scheme is said to be relative since it uses relative moments ensuring condition (\ref{eq: fully relative equilibrium matrix}).
It is important, that condition (\ref{eq: cond scalar product fully relative}) is still a strong condition.
\end{remark}

\section{Partially Relative Schemes}
\label{sec: Partially}

As mentioned in Remark \ref{re: Relative Schemes}, condition (\ref{eq: cond scalar product fully relative}) is a strong condition not providing a direct approach towards construction collision operators admitting a stability structure.
In this section, an structured approach to satisfy condition (\ref{eq: span remaining}) assuming condition (\ref{eq: span prescribed}) is fulfilled.
Here, we use that the order of consistency of the scheme is determined by the correctness of the first $\beta+\gamma$ equilibrium moments only.
The remaining equilibrium moments only influence the error term itself but not the order of the error term.
Hence, one can choose the remaining $n-\beta-\gamma$ equilibrium moments arbitrary improving stability without reducing order of consistency.
Again, assume the matrix $\mat{E}^{(3)}$ mapping conserved moments onto equilibrium moments to be of the form:
\begin{align}
\label{eq: relative equilibrium matrix}
\mat{E}^{(3)} = \begin{pmatrix}
\Id_\gamma & \zero_{\gamma \times \beta} & \zero_{\gamma \times (n-\beta-\gamma)}\\
\zero_{\beta \times \gamma} & \zero_{\beta \times \beta} & \zero_{\beta \times (n-\beta-\gamma)}\\
\zero_{(n-\beta-\gamma) \times \gamma} & \zero_{(n-\beta-\gamma) \times \beta} & \zero_{(n-\beta-\gamma) \times (n-\beta-\gamma)}
\end{pmatrix}.
\end{align}
In addition, assume a moment matrix $\mat{M}^{(3)}$ with:
\begin{align}
\label{eq: relative moment matrix}
\mat{M}^{(3)} = \begin{pmatrix}
\Id_\gamma & \zero_{\gamma \times \beta} & \zero_{\gamma \times (n-\beta-\gamma)}\\
-\mat{E}^{(1)}_{21} & \Id_\beta & \zero_{\beta \times (n-\beta-\gamma)}\\
\zero_{(n-\beta-\gamma) \times \gamma} & \zero_{(n-\beta-\gamma) \times \beta} & \Id_{(n-\beta-\gamma)}
\end{pmatrix}\ \mat{M}^{(1)}.
\end{align}
Here, only the rows generating the $\beta$ additional moments necessary for conistency are altered w.r.t. the original BGK moment matrix $\mat{M}^{(1)}$.
Since the shape of only the first $\beta+\gamma$ rows of $\mat{M}^{(3)}$ is relevant for the order of consistency, one can now alter the remaining $n-\beta-\gamma$ rows of $\mat{M}^{(3)}$ freely as long as matrix $\mat{M}^{(3)}$ remains regular.
Using this, Theorem \ref{th: partially relative schemes} can be formulated.
\begin{theorem}[Partially Relative Schemes]
\label{th: partially relative schemes}
Assume a linear LBM BGK scheme 
\begin{align*}
	f_i(t+\Delta t,\vec{x} + \Delta t \vec{c}_i) = f_i(t,\vec{x}) + \frac{1}{\tau} \vec{e}_i^T\ \left( \mat{E}^{(1)} - \Id \right)\ \vec{f}(t,\vec{x}),
\end{align*}
with moment matrix $\mat{M}^{(1)}$ and equilibrium matrix $\mat{E}^{(1)}$ which is consistent to $\alpha$-th order.
In addition, assume the equilibrium matrix $\mat{E}^{(3)}$ defined as (\ref{eq: relative equilibrium matrix}), the moment matrix $\mat{M}^{(3)}$ defined as (\ref{eq: relative moment matrix}), and a diagonal and positive definite matrix $\mat{\Lambda}^{(3)}$ fulfilling the following equivalent conditions:
\begin{align}
<\vec{r}^{(3)}_i,\vec{r}^{(3)}_j>_{\mat{\Lambda}^{(3)}} &= 0,\ \text{for}\ i=1,..,\gamma,\ j=1+\gamma,..,\beta+\gamma,\label{eq: orthogonality constraint partially}\\
\Span\left\{ \vec{r}_i^{(3)}: i = 1+\beta,\dots,\beta+\gamma \right\} &\subset \Span\left\{ \vec{r}_i^{(3)}: i = 1,\dots,\gamma \right\}^{\perp_{\mat{\Lambda}^{(3)}}},\label{eq: orthogonality constraint partially span}
\end{align}
with $\vec{r}^{(3)}_i$ denoting the $i$-th row of $\mat{M}^{(3)}$, the scalar product $<\cdot,\cdot>_{\mat{\Lambda}^{(3)}}$ defined analog to (\ref{eq: scalar product}), and superscript $\perp_{\mat{\Lambda}^{(3)}}$ denoting the orthogonal complement of the set w.r.t. this scalar product.
Then, the collision matrix 
\begin{align}
\label{eq: orthogonality theorem}
\tilde{\mat{J}}^{(3)} = \frac{1}{\tau} \tilde{\mat{M}}^{(3)-1}\ \left( \mat{E}^{(3)} - \Id_n \right)\ \tilde{\mat{M}}^{(3)}
\end{align}
with the modified moment matrix $\tilde{\mat{M}}^{(3)}$ with rows $\tilde{\vec{r}}^{(3)}_i$:
\begin{align}
\label{eq: modified moment matrix}
\tilde{\vec{r}}^{(3)} = \begin{cases}
\vec{r}^{(3)}_i, & \text{for}\ i=1,..,\beta+\gamma,\\
\vec{r}^{(3)}_i - \sum\limits_{j=1}^\gamma \frac{<\vec{r}^{(3)}_i,\vec{o}^{(3)}_j>_{\mat{\Lambda}^{(3)}}}{<\vec{o}^{(3)}_j,\vec{o}^{(3)}_j>_{\mat{\Lambda}^{(3)}}} \vec{r}^{(3)}_j, & \text{for}\ i=1+\beta+\gamma,..,n,
\end{cases}
\end{align}
and an orthogonal basis $\left\{ \vec{o}^{(3)}_i: i=1,..,\gamma\right\}$ of $\Span\left\{ \vec{r}^{(3)}_i: i=1,..,\gamma\right\}$ admits a pre-stability structure.
The LBM scheme defined by collision matrix $\tilde{\mat{J}}^{(3)}$ is consisten to $\alpha$-th order.
If in addition, $\tau \in \Big[ \frac{1}{2}, \infty \Big)$, then the collision matrix $\mat{J}^{(2)}$ admits a stability structure.
\end{theorem}
\begin{proof}
From orthogonality condition (\ref{eq: orthogonality theorem}) follows:
\begin{align*}
\tilde{\mat{\Lambda}}^{(3)}_{12} = \left( <\tilde{\vec{r}}^{(3)}_i, \tilde{\vec{r}}^{(3)}_j>_{\mat{\Lambda}^{(3)}} \right)_{i=1,..,\gamma, j=1+\gamma,..,\beta+\gamma} = \zero_{\gamma \times \beta}.
\end{align*}
The definition of rows $\tilde{\vec{r}}^{(3)}_i$ for $i=1+\beta+\gamma,..,n$ in (\ref{eq: modified moment matrix}) represents a truncated Gram-Schmidt orthogonalization.
Hence, we find:
\begin{align*}
	\Span\left\{ \vec{r}_i^{(3)}: i = 1+\beta+\gamma,\dots,n \right\} &\subset \Span\left\{ \vec{r}_i^{(3)}: i = 1,\dots,\gamma \right\}^{\perp_{\mat{\Lambda}^{(3)}}},
\end{align*}
and due to the definition of matrix $\tilde{\mat{\Lambda}}^{(3)}_{12}$:
\begin{align*}
\tilde{\mat{\Lambda}}^{(3)}_{12} = \left( <\tilde{\vec{r}}^{(3)}_i, \tilde{\vec{r}}^{(3)}_j>_{\mat{\Lambda}^{(3)}} \right)_{i=1,..,\gamma, j=1+\beta+\gamma,..,n} = \zero_{\gamma \times (n-\beta-\gamma)}.
\end{align*}
Therefore, the equivalent of condition (\ref{eq: condition general plugged in fully relative}) for $\tilde{\mat{M}}^{(3)}$, $\mat{\Lambda}^{(3)}$, and $\mat{E}^{(3)}$ is fulfilled.
By this, it is shown that $\mat{\Lambda}^{(3)}$ symmetrizes the modified collision matrix $\tilde{\mat{J}}^{(3)}$:
\begin{align*}
\tilde{\mat{J}}^{(3)}\ \mat{\Lambda}^{(3)} = \mat{\Lambda}^{(3)}\ \tilde{\mat{J}}^{(3)T}.
\end{align*}
Thus, $\tilde{\mat{J}}^{(3)}$ admits a pre-stability structure.
The proof for that this pre-stability structure is a stability structure for $\tau \in \Big[ \frac{1}{2}, \infty \Big)$ is analog to Theorem \ref{th: Fully Relative}.
Finally, since all moments necessary for ensuring $\alpha$-th order of consistency are unchanged, i.e. the first $\beta+\gamma$ rows of $\mat{M}^{(3)}$ and $\tilde{\mat{M}}^{(3)}$ coincide, the LBM with the modified collision matrix $\tilde{\mat{J}}^{(3)}$ is consistent to $\alpha$-th order.
\end{proof}
\begin{remark}[]
Theorem \ref{th: partially relative schemes} extends the idea of Theorem \ref{th: Fully Relative} and fully relative schemes to the wider class we call partially relative schemes.
The name is chosen to indicate that all moments necessary for ensuring the consistency order are chosen relative but the remaining moments are adapted for stability.
Due to these additional degrees of freedom, the partially relative schemes are applicable to wider set of linear collision operators and discrete velocity sets.
\end{remark}
\begin{remark}[The Equilibrium Distribution]
\label{re: eq distr}
Due to the specific structure of the partially relative schemes, the information the equilibrium \pdfs\ can be simplified as follows:
\begin{align*}
\vec{f}^{eq} &= \tilde{\mat{M}}^{(3)-1}\ \mat{E}^{(3)}\ \tilde{\mat{M}}^{(3)}\ \vec{f}\\
&= \begin{pmatrix}
\tilde{\mat{M}}^{(3)\dagger}_{11} & \tilde{\mat{M}}^{(3)\dagger}_{12} & \tilde{\mat{M}}^{(3)\dagger}_{13}\\
\tilde{\mat{M}}^{(3)\dagger}_{21} & \tilde{\mat{M}}^{(3)\dagger}_{22} & \tilde{\mat{M}}^{(3)\dagger}_{23}\\
\tilde{\mat{M}}^{(3)\dagger}_{31} & \tilde{\mat{M}}^{(3)\dagger}_{32} & \tilde{\mat{M}}^{(3)\dagger}_{33}
\end{pmatrix}\ \begin{pmatrix}
\Id_\gamma\\
& \zero_{\beta \times \beta}\\
& & \zero_{(n-\beta-\gamma) \times (n-\beta-\gamma)}
\end{pmatrix}\ \begin{pmatrix}
\tilde{\mat{M}}^{(3)}_{11} & \tilde{\mat{M}}^{(3)}_{12} & \tilde{\mat{M}}^{(3)}_{13}\\
\tilde{\mat{M}}^{(3)}_{21} & \tilde{\mat{M}}^{(3)}_{22} & \tilde{\mat{M}}^{(3)}_{23}\\
\tilde{\mat{M}}^{(3)}_{31} & \tilde{\mat{M}}^{(3)}_{32} & \tilde{\mat{M}}^{(3)}_{33}
\end{pmatrix}\\
&=\begin{pmatrix}
\tilde{\mat{M}}^{(3)\dagger}_{11}\\
\tilde{\mat{M}}^{(3)\dagger}_{21}\\
\tilde{\mat{M}}^{(3)\dagger}_{31}
\end{pmatrix}\ \vec{m}^{(cons)}
\end{align*}
Hence, the equilibrium \pdfs\ $\vec{f}^{eq}$ can be written in terms of the conserved moments.
Calculating the equilibrium matrix therefore only requires $2 \gamma n$ multiplications and $\gamma (n-1) + n (\gamma-1)$ additions.
\end{remark}
\begin{remark}[Implementation of the Modified Collision Operator]
Based on Remark \ref{re: eq distr}, the result of the modified collision operator $\tilde{\mat{J}}^{(3)}$ reads:
\begin{align*}
\tilde{\mat{J}}^{(3)}\ \vec{f} &= \frac{1}{\tau} \tilde{\mat{M}}^{(3)-1}\ \left( \mat{E} - \Id_n \right)\ \tilde{\mat{M}}^{(3)}\ \vec{f}\\
&= \frac{1}{\tau} \left( \begin{pmatrix}
\tilde{\mat{M}}^{(3)\dagger}_{11}\\
\tilde{\mat{M}}^{(3)\dagger}_{21}\\
\tilde{\mat{M}}^{(3)\dagger}_{31}
\end{pmatrix}\ \vec{m}^{(cons)} - \vec{f} \right).
\end{align*}
Thus, the complexity of computing the collision step is similar to the complexity of standard BGK models.
\end{remark}
\begin{remark}[]
\label{re: non-exhaustiveness}
Using the proposed structural approach, only a subset of all linear BGK-type collision operators can be constructed.
This is due to the two basic assumptions employed: 1. that the rows of the modified moment matrix $\tilde{\mat{M}}^{(3)}$ are constructed according to equation (\ref{eq: modified moment matrix}) (Theorem \ref{th: partially relative schemes}); and 2. that the equilibrium moments related to the moments generated by rows $\tilde{\vec{r}}^{(3)}_i$ for $i=1+\beta+\gamma,..,n$ of modified moment matrix $\tilde{\mat{M}}^{(3)}$ are set to $0$.
\end{remark}

\section{Application to the 3D Linearized Euler Equations}
\label{sec: 3D LEE}

The isothermal linearized Euler equations (LEE) are derived from the isothermal Euler equations:
\begin{align*}
\partial_t \rho + \nabla_\vec{x} \cdot \left( \rho \vec{u} \right) &= 0,\\
\partial_t \left( \rho \vec{u} \right) + \nabla_\vec{x} \cdot \left( \rho \vec{u} \otimes \vec{u} + c_s^2 \rho \Id \right) &= \vec{0},
\end{align*}
where $\rho$, $\vec{u}$, and $c_s$ describe density, velocity, and the speed of sound and $\otimes$ denotes the dyadic product.
The first equation represents conservation of mass and the second conservation of momentum.
The LEE are now obtained by linearization of the isothermal Euler equations around a background flow with density $\rho_0$ and velocity $\vec{u}_0$.
Let $\epsilon >0$ denote a linearization factor representing the scale of the fluctuations in density ($\rho'$) and velocity ($\vec{u}'$) around the background flow.
One then obtains:
\begin{align*}
\rho = \rho_0 + \epsilon \rho',\ \vec{u} = \vec{u}_0 + \epsilon \vec{u}'.
\end{align*}
The LEE then read:
\begin{align*}
\partial_t \rho' + \nabla_\vec{x} \cdot \left( \rho_0 \vec{u}' + \vec{u}_0 \rho' \right) &= 0,\\
\partial_t \left( \rho_0 \vec{u}' + \vec{u}_0 \rho' \right) + \nabla_\vec{x} \cdot \left( \rho_0 \vec{u}_0 \otimes \vec{u}' + \rho_0 \vec{u}' \otimes \vec{u}_0 + \vec{u}_0 \otimes \vec{u}_0 \rho' + c_s^2 \rho' \Id_2 \right) &= \vec{0}.
\end{align*}
While the derivation of stable linear collision operator is easy to show for flows without a background velocity for both isothermal flows (cf. Otte \cite{OtteDiss}) and for the compressible flows (cf. Otte and Frank \cite{Otte2016}), this is not true for the LEE with a non-vanishing background velocity.
This is particularly due to the additional terms
\begin{align*}
\nabla_{\vec{x}} \cdot \left( \rho_0 \vec{u}_0 \otimes \vec{u}' + \rho \vec{u}' \otimes \vec{u}_0 + \rho' \vec{u}_0 \otimes \vec{u}_0 \right)
\end{align*}
present in the momentum equation.
For the 3D case, this can be seen by Taylor-expanding the Maxwell distribution
\begin{align*}
	\frac{\rho}{(2 \pi c_s^2)^\frac{3}{2}} \exp\left( - \frac{|\vec{v}-\vec{u}|^2}{2c_s^2} \right)
\end{align*}
around the non-trivial background belocity $\vec{u}_0$ which results in a discrete equilibrium matrix
\begin{align*}
\mat{E}_{pdf} = \diag(\vec{w})\ \mat{E}_{symm}
\end{align*}
with symmetric $\mat{E}_{symm}$ which lends itself well to stability structures (cf. Rheinländer \cite{Rheinlaender2010}).
Still, for this equilibrium the lattice symmetries are more involved and in particular require the absence of aliasing of third- and fourth-order moments of the weights $\vec{w}$.
Hence, prohibitively large velocity sets are required.
In contrast, a simple linearization of the common equilibrium distribution for the Navier-Stokes equations around $\rho_0$ and $\vec{u}_0$ results in a discrete equilibrium
\begin{align*}
\mat{E}_{pdf} = \diag(\vec{w})\ \mat{E}_{asymm}
\end{align*}
with asymmetrical $\mat{E}_{asymm}$ which, in general, is not suited for application of the notion of stability structures (cf. Rheinländer \cite{Rheinlaender2010}).
This initially motivated the development of partially relative schemes as presented in section \ref{sec: Partially}.
\medskip
\\
Analog to Otte and Frank \cite{Otte2016}, one finds the following consistency result for a LBM scheme in acoustic scaling solving the LEE:
\begin{theorem}[Consistency Result]
\label{th: consistency}
Assume the Lattice Boltzmann equation in acoustic scaling with $\tau = \frac{1}{2}$, a discrete velocity set $\mathcal{S} = \left\{ \vec{c}_i: i=1,..,n\right\}$, and an equilibrium function $\vec{f}^{eq}(\cdot)$ fulfilling the following conditions on its moments:
\begin{align}
\sum_{i=1}^n f_i =\sum_{i=1}^n f^{eq}\left(\vec{f}\right) &= \rho',\label{eq: mom cond 1}\\
\sum_{i=1}^n \vec{c}_i f_i= \sum_{i=1}^n \vec{c}_i f^{eq}\left(\vec{f}\right) &= \rho_0 \vec{u}' + \vec{u}_0 \rho',\label{eq: mom cond 2}\\
\sum_{i=1}^n \vec{c}_i \otimes \vec{c}_i f^{eq}\left(\vec{f}\right) &= \rho_0\vec{u}_0 \vec{u}'^T + \rho_0 \vec{u}'\vec{u}_0^T + \vec{u}_0 \vec{u}_0^T \rho' + c_s^2 \rho' \Id\label{eq: mom cond 3}
\end{align}
Then this LB scheme is consistent of second order.
\end{theorem}
\begin{proof}
The LBE in acoustic scaling is given by:
\begin{align}
\label{eq: LBE ac}
f_i(t+\Delta t,\vec{x} + \Delta t \vec{c}_i) - f_i(t,\vec{x}) &= \frac{1}{\tau} \left( f^{eq}_i(\vec{f}) - \vec{f} \right),\ \text{for}\ i=1,..,n.
\end{align}
Assume a formal expansion of $\vec{f}$, $\vec{f}^{eq}(\vec{f})$, and the macroscopic quantities $\rho'$ and $\vec{u'}$:
\begin{align*}
\vec{f} &= \sum_{k=0}^\infty \Delta t^k \vec{f}^{(k)}, & \vec{f}^{eq} &= \sum_{k=0}^\infty \Delta t^k \vec{f}^{eq(k)},\\
\rho' &= \sum_{k=0}^\infty \Delta t^k \rho'^{(k)}, & \vec{u}' &= \sum_{k=0}^\infty \Delta t^k \vec{u}'^{(k)}.
\end{align*}
For brevity, the arument of $\vec{f}^{eq}$ is dropped.
After plugging this into (\ref{eq: LBE ac}), Taylor expansion in time and space around $(t,\vec{x})$, and sorting by powers of $\Delta t$, one obtains:
\begin{align}
\O\left(\Delta t^0\right): && 0 &= \frac{1}{\tau} \left( f_i^{eq(0)} - f_i^{(0)} \right),\label{eq: LBE Taylor 1}\\
\O\left(\Delta t^1\right): && \partial_t f_i^{(0)} + \vec{c}_i \cdot \nabla_\vec{x} f_i^{(0)} &= \frac{1}{\tau} \left( f_i^{eq(1)} - f_i^{(1)} \right),\label{eq: LBE Taylor 2}\\
\O\left(\Delta t^2\right): && \partial_t f_i^{(1)} + \vec{c}_i \cdot \nabla_\vec{x} f_i^{(1)}\nonumber\\
&& + \frac{1}{2} \partial_t^2 f_i^{(0)} + \vec{c}_i \cdot \partial_t \nabla_\vec{x} f_i^{(0)} + \frac{1}{2} \vec{c}_i \otimes \vec{c}_i: \nabla_\vec{x} \otimes \nabla_\vec{x} f_i^{(0)} &= \frac{1}{\tau} \left( f_i^{eq(2)} - f_i^{(2)} \right),\label{eq: LBE Taylor 3}
\end{align}
where $:$ denotes a tensor contraction.
From equation (\ref{eq: LBE Taylor 1}) one finds: $\vec{f}^{(0)} = \vec{f}^{eq(0)}$.
With conditions (\ref{eq: mom cond 1}) and (\ref{eq: mom cond 2}), the zeroth and first order moments of equation (\ref{eq: LBE Taylor 2}) read:
\begin{align*}
\partial_t \rho'^{(0)} + \nabla_\vec{x} \cdot \left( \rho_0 \vec{u}'^{(0)} + \vec{u}_0 \rho'^{(0)} \right) &= 0, \\
\partial_t \left( \rho_0 \vec{u}'^{(0)} + \vec{u}_0 \rho'^{(0)} \right) + \nabla_\vec{x} \cdot \left( \rho_0 \vec{u}_0 \otimes \vec{u}'^{(0)} + \rho_0 \vec{u}'^{(0)} \otimes \vec{u}_0 + \vec{u}_0 \otimes \vec{u}_0 \rho'^{(0)} + c_s^2 \rho'^{(0)} \Id_2 \right) &= \vec{0}.
\end{align*}
Hence, the LB scheme is at least first order consistent.
By rearranging equation (\ref{eq: LBE Taylor 2}), one can derive a form for $f_i^{(1)}$:
\begin{align*}
f_i^{(1)} = f_i^{eq(1)} - \tau \left( \partial_t f_i^{(0)} + \vec{c}_i \cdot \nabla_\vec{x} f_i^{(0)} \right).
\end{align*}
Plugging this into equation (\ref{eq: LBE Taylor 3}), this gives:
\begin{align*}
&\partial_t f_i^{eq(1)} + \vec{c}_i \cdot \nabla_{\vec{x}} f_i^{eq(1)}\\
+& \left( \frac{1}{2} - \tau \right) \left( \partial_t^2 f_i^{eq(0)} + 2 \vec{c}_i \cdot \partial_t \nabla_\vec{x} f_i^{eq(0)}+ \vec{c}_i \otimes \vec{c}_i : \nabla_\vec{x} \otimes \nabla_\vec{x} f_i^{eq(2)} \right) &= \frac{1}{\tau} \left( f^{eq(2)}_i - f^{(2)}_i \right).
\end{align*}
With $\tau = \frac{1}{2}$, the term involving $f^{eq(0)}_i$ vanishes and therefore the zeroth and first order moments are given by:
\begin{align*}
\partial_t \rho'^{(1)} + \nabla_\vec{x} \cdot \left( \rho_0 \vec{u}'^{(1)} + \vec{u}_0 \rho'^{(1)} \right) &= 0, \\
\partial_t \left( \rho_0 \vec{u}'^{(1)} + \vec{u}_0 \rho'^{(1)} \right) + \nabla_\vec{x} \cdot \left( \rho_0 \vec{u}_0 \otimes \vec{u}'^{(1)} + \rho_0 \vec{u}'^{(1)} \otimes \vec{u}_0 + \vec{u}_0 \otimes \vec{u}_0 \rho'^{(1)} + c_s^2 \rho'^{(1)} \Id_2 \right) &= \vec{0}.
\end{align*}
Hence, the LB scheme is of second order consistency.
\end{proof}
\begin{remark}[]
As shown in Theorem \ref{th: consistency}, for proof of consistency it suffices to assume a discrete velocity set $\mathcal{S}$ for which an equilibrium distribution function $\vec{f}^{eq}(\cdot)$ satisfying conditions (\ref{eq: mom cond 1})--(\ref{eq: mom cond 3}) exists.
\end{remark}
Based on this general consistency result, it is now shown that for the D3Q33 velocity set with $c_s^2 = \frac{1}{3}$ partially relative schemes exist admitting a stability structure.
Note that for velocity sets smaller than D3Q33, i.e. the D3Q27 velocity set and its symmetric subsets D3Q7, D3Q9, D3Q13, D3Q15, D3Q19, and D3Q21, linear program (\ref{eq: linear program}) does not assume a feasible solution.
The D3Q33 velocity set is given by:
\begin{align*}
\mathcal{S} = \big\{ & \vec{c}_1 = (0,0,0)^T,\\
& \vec{c}_2 = (-1,0,0)^T, \vec{c}_3 = (1,0,0)^T, \vec{c}_4 = (0,-1,0)^T, \vec{c}_5 = (0,1,0)^T, \vec{c}_6 - (0,0,-1)^T, \vec{c}_7 = (0,0,1)^T\\
& \vec{c}_8 = (-1,-1,0)^T, \vec{c}_9 = (-1,1,0)^T, \vec{c}_{10} = (1,-1,0)^T, \vec{c}_{11} = (1,1,0)^T,\\
& \vec{c}_{12} = (-1,0,-1)^T, \vec{c}_{13} = (-1,0,1)^T, \vec{c}_{14} = (1,0,-1)^T, \vec{c}_{15} = (1,0,1)^T,\\
& \vec{c}_{16} = (0,-1,-1)^T, \vec{c}_{17} = (0,-1,1)^T, \vec{c}_{18} = (0,1,-1)^T, \vec{c}_{19} = (0,1,1)^T,\\
& \vec{c}_{20} = (-1,-1,-1)^T, \vec{c}_{21} = (-1,-1,1)^T, \vec{c}_{22} = (-1,1,-1)^T, \vec{c}_{23} = (-1,1,1)^T,\\
& \vec{c}_{24} = (1,-1,-1)^T, \vec{c}_{25} = (1,-1,1)^T, \vec{c}_{26} = (1,1,-1)^T, \vec{c}_{27} = (1,1,1)^T,\\
& \vec{c}_{28} = (-2,0,0)^T, \vec{c}_{29} = (2,0,0)^T, \vec{c}_{30} = (0,-2,0)^T, \vec{c}_{31} = (0,2,0)^T, \vec{c}_{32} = (0,0,-2)^T, \vec{c}_{33} = (0,0,2)^T \big\}.
\end{align*}
Assume the moment matrix $\mat{M}^{(1)}$ where the first row maps the \pdfs\ onto density, the following three rows onto momentum, the next six rows onto the second moments, and the remaining 23 rows onto all unique raw third, fourth, fifth, and sixth order moments.
The moment matrix $\mat{M}^{(1)}$ is shown in Figure \ref{fig: Matrix M0}.
The partially relative moment matrix $\mat{M}^{(3)}$ is obtained according to (\ref{eq: relative moment matrix}).
The orthogonality constraint (\ref{eq: orthogonality constraint partially}) reads:
\begin{align*}
\vec{r}^{(3)}_i\ \mat{\Lambda}^{(3)}\ \vec{r}^{(3)T}_j \overset{!}{=} 0,\ \text{for}\ i=1,..,4, j=5,..,10.
\end{align*}
Since all terms in these 24 conditions are linear w.r.t. the diagonal entries of $\mat{\Lambda}^{(3)}$, these conditions can be written in form of a linear system:
\begin{align}
\mat{A}\ \vec{\lambda}^{(3)} \overset{!}{=} \vec{0}_{24},\label{eq: kernel condition}
\end{align}
where $\vec{\lambda}^{(3)}$ denotes the vector containing the diagonal entries of $\mat{\Lambda}^{(3)}$.
Hence, any solution to condition (\ref{eq: orthogonality constraint partially}) must be an element of the kernel $\ker(\mat{A})$ of $\mat{A}$.
Since $\Dim\left( \ker(\mat{A}) \right) = 14$, each element of the kernel can be described by a basis of the kernel and 14 coefficients $\alpha_i$ for $i=1,..,14$.
An element of the kernel has been computed using \textit{Mathematica} but its form is too complicated for prcatical use.
Hence, it is not presented here.
Due to Theorem \ref{th: partially relative schemes}, the modified collision operator $\tilde{\mat{J}}^{(3)}$ defined in equation (\ref{eq: orthogonality theorem}) admits a pre-stability if and only if matrix $\mat{\Lambda}^{(3)}$ is diagonal and positive definite.
The matrix $\mat{\Lambda}^{(3)}$ is diagonal and positive definite if and only if a vector $\vec{k} \in \ker(\mat{A})$ exists which lies in the positive orthant, i.e. $k_i >0$ for $i=1,..,n$.
Figure \ref{fig: stability domain} shows combinations of $u_{01}$, $u_{02}$, and $u_{03}$ for which such a $\vec{k}$ exists.
Assuming the kernel of matrix $\mat{A}$ contains an element within the positive orthant, it contains infinitely many such elements.
In order to choose an element within the intersection of $\ker(\mat{A})$ and the positive orthant which does not contain excessively large components, we replace condition (\ref{eq: kernel condition}) by the linear program:
\begin{align}
\left.
\begin{aligned}
\min_{\vec{\lambda} \in \R^n} & \sum_{i=1}^n \lambda_i\\
\text{s.t.}\ & \vec{\lambda} \in \ker(\mat{A})\\
& \lambda_i \geq 1\ \text{for}\ i=1,\dots,n.
\end{aligned}
\right\}\label{eq: linear program}
\end{align}
The objective function is chosen such that the magnitude of the components of $\vec{\lambda}$ is controlled while condition $\lambda_i \geq 1$ for $i=1,\dots,n$ is used to avoid solutions with components almost $0$.
Since for every $\vec{\lambda} \in \ker(\mat{A})$ also $\alpha \vec{\lambda} \in \ker(\mat{A})$ for $\alpha \in R$, condition $\lambda_i \geq 1$ for $i=1,\dots,n$ does not impose an additional constraint on the existence of a solution compared to condition (\ref{eq: kernel condition}).
Hence, every feasible solution to linear program (\ref{eq: linear program}) also is a solution to kernel condition (\ref{eq: kernel condition}).
\begin{remark}[]
While, the consistency result of Theorem \ref{th: consistency} guarantees second-order consistency, the structure of the error term depends on the choice of background velocities.
This is due to the process of constructing rows $\tilde{\vec{r}}_i^{(3)}$ for $i=1+\beta+\gamma,..,n$ of the modified moment matrix $\tilde{\mat{M}}^{(3)}$ (compare (\ref{eq: modified moment matrix}) and the algorithm of choosing the diagonal and positive definite matrix $\mat{\Lambda}^{(3)}$ (compare linear program (\ref{eq: linear program})).
Hence, Galilean invariance of the resulting methods w.r.t. the background velocity $\vec{u}_0$ is not ensured.
\end{remark}
In Figures \ref{fig: stab dom 1 6} and \ref{fig: stab dom 1 3}, we present the domain within the plane spanned by the second and third components of background velocity $\vec{u}_0$ for which linear program (\ref{eq: linear program}) admits a feasible solution for fixed values $\frac{1}{6}$ and $\frac{1}{3}$ of the first component of $\vec{u}_0$.
Note that no feasible solution to linear program (\ref{eq: linear program}) could be found for background velocities with at least one vanishing component.
This effect is due to the non-exhaustiveness of the proposed approach as noted in Remark \ref{re: non-exhaustiveness} in combination with the particular structure of the feasibility condition (\ref{eq: kernel condition}) also found in the linear program (\ref{eq: linear program}).
Still, we find that the proposed approach provides a helpful tool to the construction of stable linear collision operators for a wide variety of background velocities.
\begin{sidewaysfigure}
\begin{small}
\begin{align*}
\mat{M}^{(1)} = \left(
\begin{array}{ccccccccccccccccccccccccccccccccc}
 1 & 1 & 1 & 1 & 1 & 1 & 1 & 1 & 1 & 1 & 1 & 1 & 1 & 1 & 1 & 1 & 1 & 1 & 1 & 1 & 1 & 1 & 1 & 1 & 1 & 1 & 1 & 1 & 1 & 1 & 1 & 1 & 1 \\
 0 & -1 & 1 & 0 & 0 & 0 & 0 & -1 & -1 & 1 & 1 & -1 & -1 & 1 & 1 & 0 & 0 & 0 & 0 & -1 & -1 & -1 & -1 & 1 & 1 & 1 & 1 & -2 & 2 & 0 & 0 & 0 & 0 \\
 0 & 0 & 0 & -1 & 1 & 0 & 0 & -1 & 1 & -1 & 1 & 0 & 0 & 0 & 0 & -1 & -1 & 1 & 1 & -1 & -1 & 1 & 1 & -1 & -1 & 1 & 1 & 0 & 0 & -2 & 2 & 0 & 0 \\
 0 & 0 & 0 & 0 & 0 & -1 & 1 & 0 & 0 & 0 & 0 & -1 & 1 & -1 & 1 & -1 & 1 & -1 & 1 & -1 & 1 & -1 & 1 & -1 & 1 & -1 & 1 & 0 & 0 & 0 & 0 & -2 & 2 \\
 0 & 1 & 1 & 0 & 0 & 0 & 0 & 1 & 1 & 1 & 1 & 1 & 1 & 1 & 1 & 0 & 0 & 0 & 0 & 1 & 1 & 1 & 1 & 1 & 1 & 1 & 1 & 4 & 4 & 0 & 0 & 0 & 0 \\
 0 & 0 & 0 & 0 & 0 & 0 & 0 & 1 & -1 & -1 & 1 & 0 & 0 & 0 & 0 & 0 & 0 & 0 & 0 & 1 & 1 & -1 & -1 & -1 & -1 & 1 & 1 & 0 & 0 & 0 & 0 & 0 & 0 \\
 0 & 0 & 0 & 0 & 0 & 0 & 0 & 0 & 0 & 0 & 0 & 1 & -1 & -1 & 1 & 0 & 0 & 0 & 0 & 1 & -1 & 1 & -1 & -1 & 1 & -1 & 1 & 0 & 0 & 0 & 0 & 0 & 0 \\
 0 & 0 & 0 & 1 & 1 & 0 & 0 & 1 & 1 & 1 & 1 & 0 & 0 & 0 & 0 & 1 & 1 & 1 & 1 & 1 & 1 & 1 & 1 & 1 & 1 & 1 & 1 & 0 & 0 & 4 & 4 & 0 & 0 \\
 0 & 0 & 0 & 0 & 0 & 0 & 0 & 0 & 0 & 0 & 0 & 0 & 0 & 0 & 0 & 1 & -1 & -1 & 1 & 1 & -1 & -1 & 1 & 1 & -1 & -1 & 1 & 0 & 0 & 0 & 0 & 0 & 0 \\
 0 & 0 & 0 & 0 & 0 & 1 & 1 & 0 & 0 & 0 & 0 & 1 & 1 & 1 & 1 & 1 & 1 & 1 & 1 & 1 & 1 & 1 & 1 & 1 & 1 & 1 & 1 & 0 & 0 & 0 & 0 & 4 & 4 \\
 0 & -1 & 1 & 0 & 0 & 0 & 0 & -1 & -1 & 1 & 1 & -1 & -1 & 1 & 1 & 0 & 0 & 0 & 0 & -1 & -1 & -1 & -1 & 1 & 1 & 1 & 1 & -8 & 8 & 0 & 0 & 0 & 0 \\
 0 & 0 & 0 & -1 & 1 & 0 & 0 & -1 & 1 & -1 & 1 & 0 & 0 & 0 & 0 & -1 & -1 & 1 & 1 & -1 & -1 & 1 & 1 & -1 & -1 & 1 & 1 & 0 & 0 & -8 & 8 & 0 & 0 \\
 0 & 0 & 0 & 0 & 0 & -1 & 1 & 0 & 0 & 0 & 0 & -1 & 1 & -1 & 1 & -1 & 1 & -1 & 1 & -1 & 1 & -1 & 1 & -1 & 1 & -1 & 1 & 0 & 0 & 0 & 0 & -8 & 8 \\
 0 & 0 & 0 & 0 & 0 & 0 & 0 & -1 & 1 & -1 & 1 & 0 & 0 & 0 & 0 & 0 & 0 & 0 & 0 & -1 & -1 & 1 & 1 & -1 & -1 & 1 & 1 & 0 & 0 & 0 & 0 & 0 & 0 \\
 0 & 0 & 0 & 0 & 0 & 0 & 0 & 0 & 0 & 0 & 0 & -1 & 1 & -1 & 1 & 0 & 0 & 0 & 0 & -1 & 1 & -1 & 1 & -1 & 1 & -1 & 1 & 0 & 0 & 0 & 0 & 0 & 0 \\
 0 & 0 & 0 & 0 & 0 & 0 & 0 & -1 & -1 & 1 & 1 & 0 & 0 & 0 & 0 & 0 & 0 & 0 & 0 & -1 & -1 & -1 & -1 & 1 & 1 & 1 & 1 & 0 & 0 & 0 & 0 & 0 & 0 \\
 0 & 0 & 0 & 0 & 0 & 0 & 0 & 0 & 0 & 0 & 0 & 0 & 0 & 0 & 0 & -1 & 1 & -1 & 1 & -1 & 1 & -1 & 1 & -1 & 1 & -1 & 1 & 0 & 0 & 0 & 0 & 0 & 0 \\
 0 & 0 & 0 & 0 & 0 & 0 & 0 & 0 & 0 & 0 & 0 & -1 & -1 & 1 & 1 & 0 & 0 & 0 & 0 & -1 & -1 & -1 & -1 & 1 & 1 & 1 & 1 & 0 & 0 & 0 & 0 & 0 & 0 \\
 0 & 0 & 0 & 0 & 0 & 0 & 0 & 0 & 0 & 0 & 0 & 0 & 0 & 0 & 0 & -1 & -1 & 1 & 1 & -1 & -1 & 1 & 1 & -1 & -1 & 1 & 1 & 0 & 0 & 0 & 0 & 0 & 0 \\
 0 & 0 & 0 & 0 & 0 & 0 & 0 & 0 & 0 & 0 & 0 & 0 & 0 & 0 & 0 & 0 & 0 & 0 & 0 & -1 & 1 & 1 & -1 & 1 & -1 & -1 & 1 & 0 & 0 & 0 & 0 & 0 & 0 \\
 0 & 1 & 1 & 0 & 0 & 0 & 0 & 1 & 1 & 1 & 1 & 1 & 1 & 1 & 1 & 0 & 0 & 0 & 0 & 1 & 1 & 1 & 1 & 1 & 1 & 1 & 1 & 16 & 16 & 0 & 0 & 0 & 0 \\
 0 & 0 & 0 & 1 & 1 & 0 & 0 & 1 & 1 & 1 & 1 & 0 & 0 & 0 & 0 & 1 & 1 & 1 & 1 & 1 & 1 & 1 & 1 & 1 & 1 & 1 & 1 & 0 & 0 & 16 & 16 & 0 & 0 \\
 0 & 0 & 0 & 0 & 0 & 1 & 1 & 0 & 0 & 0 & 0 & 1 & 1 & 1 & 1 & 1 & 1 & 1 & 1 & 1 & 1 & 1 & 1 & 1 & 1 & 1 & 1 & 0 & 0 & 0 & 0 & 16 & 16 \\
 0 & 0 & 0 & 0 & 0 & 0 & 0 & 1 & 1 & 1 & 1 & 0 & 0 & 0 & 0 & 0 & 0 & 0 & 0 & 1 & 1 & 1 & 1 & 1 & 1 & 1 & 1 & 0 & 0 & 0 & 0 & 0 & 0 \\
 0 & 0 & 0 & 0 & 0 & 0 & 0 & 0 & 0 & 0 & 0 & 1 & 1 & 1 & 1 & 0 & 0 & 0 & 0 & 1 & 1 & 1 & 1 & 1 & 1 & 1 & 1 & 0 & 0 & 0 & 0 & 0 & 0 \\
 0 & 0 & 0 & 0 & 0 & 0 & 0 & 0 & 0 & 0 & 0 & 0 & 0 & 0 & 0 & 1 & 1 & 1 & 1 & 1 & 1 & 1 & 1 & 1 & 1 & 1 & 1 & 0 & 0 & 0 & 0 & 0 & 0 \\
 0 & 0 & 0 & 0 & 0 & 0 & 0 & 0 & 0 & 0 & 0 & 0 & 0 & 0 & 0 & 0 & 0 & 0 & 0 & 1 & -1 & -1 & 1 & 1 & -1 & -1 & 1 & 0 & 0 & 0 & 0 & 0 & 0 \\
 0 & 0 & 0 & 0 & 0 & 0 & 0 & 0 & 0 & 0 & 0 & 0 & 0 & 0 & 0 & 0 & 0 & 0 & 0 & 1 & -1 & 1 & -1 & -1 & 1 & -1 & 1 & 0 & 0 & 0 & 0 & 0 & 0 \\
 0 & 0 & 0 & 0 & 0 & 0 & 0 & 0 & 0 & 0 & 0 & 0 & 0 & 0 & 0 & 0 & 0 & 0 & 0 & 1 & 1 & -1 & -1 & -1 & -1 & 1 & 1 & 0 & 0 & 0 & 0 & 0 & 0 \\
 0 & 0 & 0 & 0 & 0 & 0 & 0 & 0 & 0 & 0 & 0 & 0 & 0 & 0 & 0 & 0 & 0 & 0 & 0 & -1 & 1 & -1 & 1 & -1 & 1 & -1 & 1 & 0 & 0 & 0 & 0 & 0 & 0 \\
 0 & 0 & 0 & 0 & 0 & 0 & 0 & 0 & 0 & 0 & 0 & 0 & 0 & 0 & 0 & 0 & 0 & 0 & 0 & -1 & -1 & 1 & 1 & -1 & -1 & 1 & 1 & 0 & 0 & 0 & 0 & 0 & 0 \\
 0 & 0 & 0 & 0 & 0 & 0 & 0 & 0 & 0 & 0 & 0 & 0 & 0 & 0 & 0 & 0 & 0 & 0 & 0 & -1 & -1 & -1 & -1 & 1 & 1 & 1 & 1 & 0 & 0 & 0 & 0 & 0 & 0 \\
 0 & 0 & 0 & 0 & 0 & 0 & 0 & 0 & 0 & 0 & 0 & 0 & 0 & 0 & 0 & 0 & 0 & 0 & 0 & 1 & 1 & 1 & 1 & 1 & 1 & 1 & 1 & 0 & 0 & 0 & 0 & 0 & 0 \\
\end{array}
\right)
\end{align*}
\end{small}
\caption{D3Q33 raw moment matrix $\mat{M}^{(1)}$}
\label{fig: Matrix M0}
\end{sidewaysfigure}

\begin{figure}
	\begin{subfigure}{0.45\textwidth}
		\includegraphics[width=\textwidth]{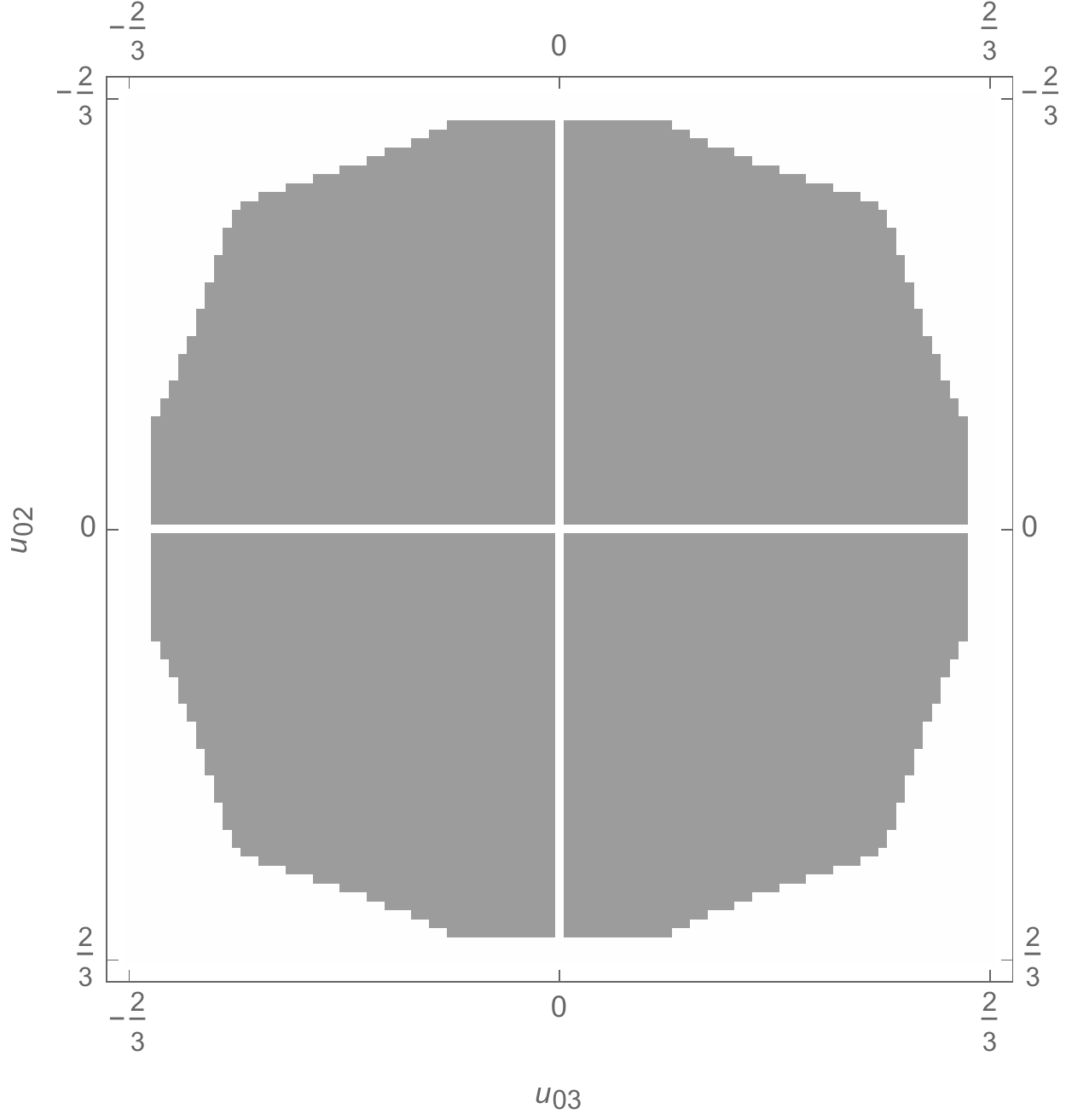}
		\caption{Approximated domain of background for velocity components two and three of $\vec{u}_0$ and $u_{01} = \frac{1}{6}$.}
		\label{fig: stab dom 1 6}
	\end{subfigure}
	\hfill
	\begin{subfigure}{0.45\textwidth}
		\includegraphics[width=\textwidth]{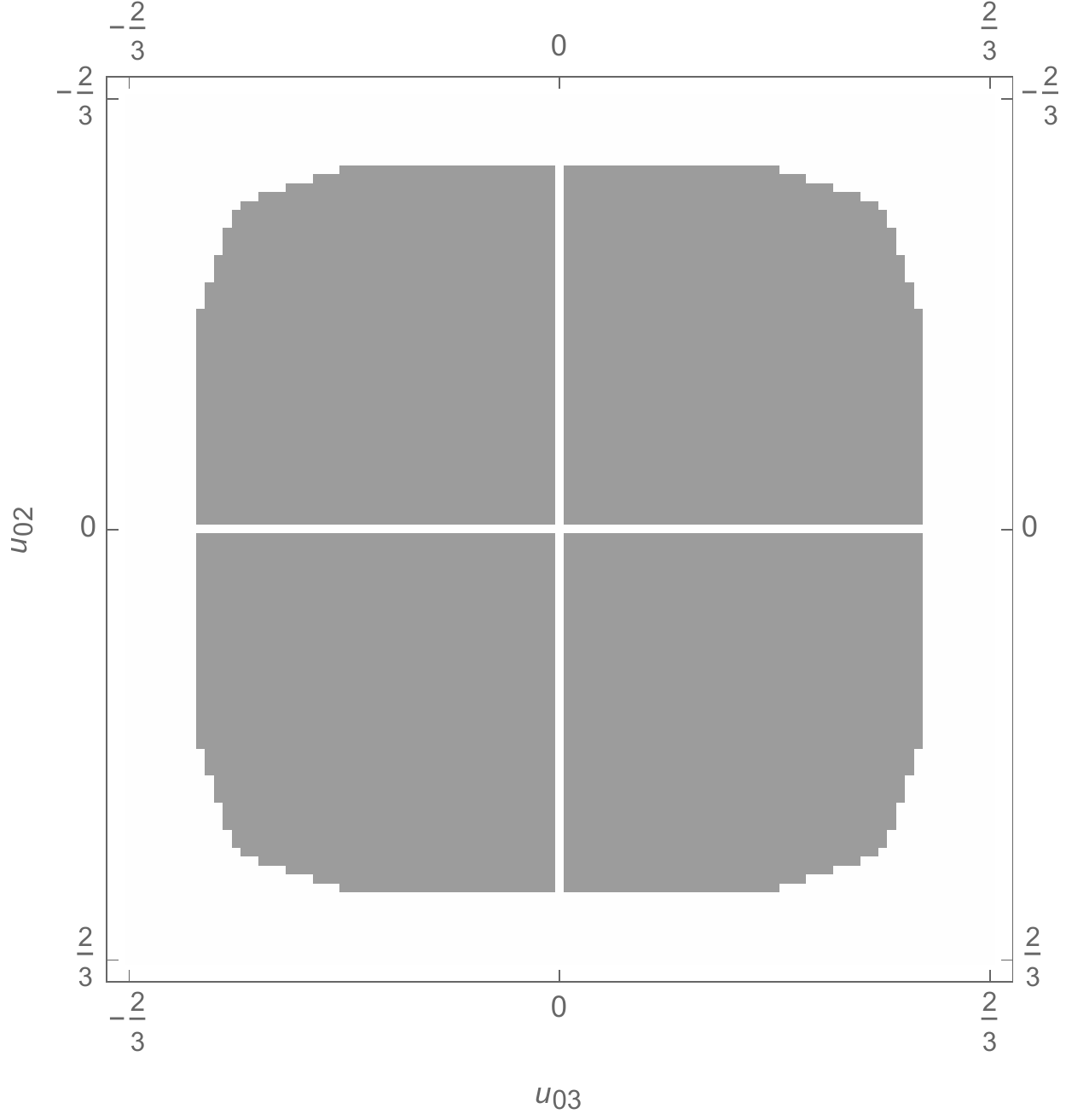}
		\caption{Approximated domain of background for velocity components two and three of $\vec{u}_0$ and $u_{01} = \frac{1}{3}$.}
		\label{fig: stab dom 1 3}
	\end{subfigure}
	\caption{Stability domains w.r.t. $u_{02}$ and $u_{03}$ for different values of $u_{01}$. Gray parts of the domain represent combinations of components two and three of background velocity $\vec{u}_0$ for fixed values of the first component for which a collision operator $\mat{J}^{(3)}$ admits a stability structure for $\tau \geq \frac{1}{2}$.}
	\label{fig: stability domain}
\end{figure}

\section{Results for the 3D Linearized Euler Equations}
\label{sec: results}

In this section, we present results for 3D flows simulated using the scheme derived in section \ref{sec: 3D LEE}.
We consider three test-cases: two pseudo-1D test cases allowing for comparison with the exact solution; and a 3D test case for which we study convergence to a highly-resolved solution.
All test cases are defined on the periodic domain $[0,1)^3$, use relaxation time $\tau = \frac{1}{2}$, and results are presented for four different background velocities $\vec{u}_0$:
\begin{itemize}
	\item $\vec{u}_0 = \frac{1}{\sqrt{3}} \left( \frac{3}{20}, \frac{1}{10}, \frac{1}{5} \right)^T$;
	\item $\vec{u}_0 = \frac{1}{\sqrt{3}} \left( -\frac{1}{4}, \frac{1}{4}, \frac{1}{2} \right)^T$;
	\item $\vec{u}_0 = \frac{1}{\sqrt{3}} \left( \frac{2}{5}, \frac{9}{10}, \frac{3}{4} \right)^T$; and
	\item $\vec{u}_0 = \frac{1}{\sqrt{3}} \left( \frac{3}{4}, \frac{5}{8}, 1 \right)^T$.
\end{itemize}
The convergence results presented in Figures \ref{fig: test case 1}--\ref{fig: test case 3} are measured in the discrete $L^\infty$-norm.
Note that the discrete $L^\infty$-norm is stricter than the discrete $L^2$-norm.
\medskip
\\
The first test case is a pseudo-1D one meaning that the initial conditions are chosen such that they depend on a single dimension only.
It represents a simple test case with only few modes present in its solution and with all modes admitting long wavelengths.
The background density $\rho_0$ is set to $\frac{2}{5}$ and the initial conditions are given as:
\begin{equation*}
	\begin{aligned}
	\rho'(t=0,x) &= \cos( 4 \pi x ),\\
	u'(t=0,x) &= \frac{5}{\sqrt{3}} \cos( 2 \pi x ),
	\end{aligned}
\end{equation*}
where variable $x$ denotes the first component of spatial vector $\vec{x}$.
From Figure \ref{fig: test case 1}, we find clear second-order convergence of the LBM using the partially relative velocity collision operator introduced in sections \ref{sec: Partially} and \ref{sec: 3D LEE}.
Note that the small initial bump is due to resolution insufficient for capturing the modes of the solution.
\medskip
\\
The second test case as well is a pseudo-1D test but with additional modes with higher wavenumber.
The background density $\rho_0$ is set to $\frac{1}{5}$ and the initial conditions are given as:
\begin{equation*}
	\begin{aligned}
	\rho'(t=0,x) &= \frac{7}{10} \sin( 2 \pi x ) \sin( 4 \pi x ) \cos( 4 \pi x ) \cos( 8 \pi x ),\\
	u'(t=0,x) &= \frac{5}{2\sqrt{3}} \sin( 8 \pi x ) \cos( 2 \pi x ),
	\end{aligned}
\end{equation*}
where variable $x$ denotes the first component of spatial vector $\vec{x}$.
The solution to this test case contains more and higher-frequency modes than that of test case 1.
Due to this, we observe in Figure \ref{fig: test case 2} that a higher minimum resolution is required before second-order convergence manifests.
Still, as soon as the resolution is sufficient for capturing the modes of the solution, we observe clear second-order convergence.
These results indicate that the derived LBM is suitable for the approximation of periodic and smooth solutions.
\medskip
\\
The third test case simulates the interference of nine Gauss-pulses located in the center of the domain and at positions $\left(\frac{10\pm 3}{20},\frac{10\pm 3}{20},\frac{10\pm 3}{20}\right)^T$.
The background density $\rho_0$ is set to $\frac{1}{5}$ and the initial conditions are given as:
\begin{equation*}
	\begin{aligned}
	\rho'(t=0,\vec{x}) &= \exp\left( -100 \left| \left| \vec{x} - \begin{pmatrix}\frac{1}{2},\frac{1}{2},\frac{1}{2}\end{pmatrix}^T \right| \right|^2 \right)\\
	&+ \exp\left( -100 \left| \left| \vec{x} - \begin{pmatrix}\frac{7}{20},\frac{7}{20},\frac{7}{20}\end{pmatrix}^T \right| \right|^2 \right)
	+ \exp\left( -100 \left| \left| \vec{x} - \begin{pmatrix}\frac{7}{20},\frac{7}{20},\frac{13}{20}\end{pmatrix}^T \right| \right|^2 \right)\\
	&+ \exp\left( -100 \left| \left| \vec{x} - \begin{pmatrix}\frac{7}{20},\frac{13}{20},\frac{7}{20}\end{pmatrix}^T \right| \right|^2 \right)
	+ \exp\left( -100 \left| \left| \vec{x} - \begin{pmatrix}\frac{7}{20},\frac{13}{20},\frac{13}{20}\end{pmatrix}^T \right| \right|^2 \right)\\
	&+ \exp\left( -100 \left| \left| \vec{x} - \begin{pmatrix}\frac{13}{20},\frac{7}{20},\frac{7}{20}\end{pmatrix}^T \right| \right|^2 \right)
	+ \exp\left( -100 \left| \left| \vec{x} - \begin{pmatrix}\frac{13}{20},\frac{7}{20},\frac{13}{20}\end{pmatrix}^T \right| \right|^2 \right)\\
	&+ \exp\left( -100 \left| \left| \vec{x} - \begin{pmatrix}\frac{13}{20},\frac{13}{20},\frac{7}{20}\end{pmatrix}^T \right| \right|^2 \right)
	+ \exp\left( -100 \left| \left| \vec{x} - \begin{pmatrix}\frac{13}{20},\frac{13}{20},\frac{13}{20}\end{pmatrix}^T \right| \right|^2 \right),\\
	\vec{u}'(t=0,\vec{x}) &= \vec{0}.
	\end{aligned}
\end{equation*}
Again, analyzing Figure \ref{fig: test case 3}, we observe that a minimum resolution is required to capture the modes of the solution.
From this point on, we find clear second-order convergence.
This indicates that the scheme derive in section \ref{sec: 3D LEE} is capable of solving the isothermal linearized Euler equations for initial conditions containing steep gradients.
\medskip
\\
The results presented in this section show that the structured approach described in section \ref{sec: relative} and \ref{sec: Partially} suitable for the derivation of stable linear collision operators.

\begin{figure}
	\begin{subfigure}{0.45\textwidth}
		\includegraphics[width=\textwidth]{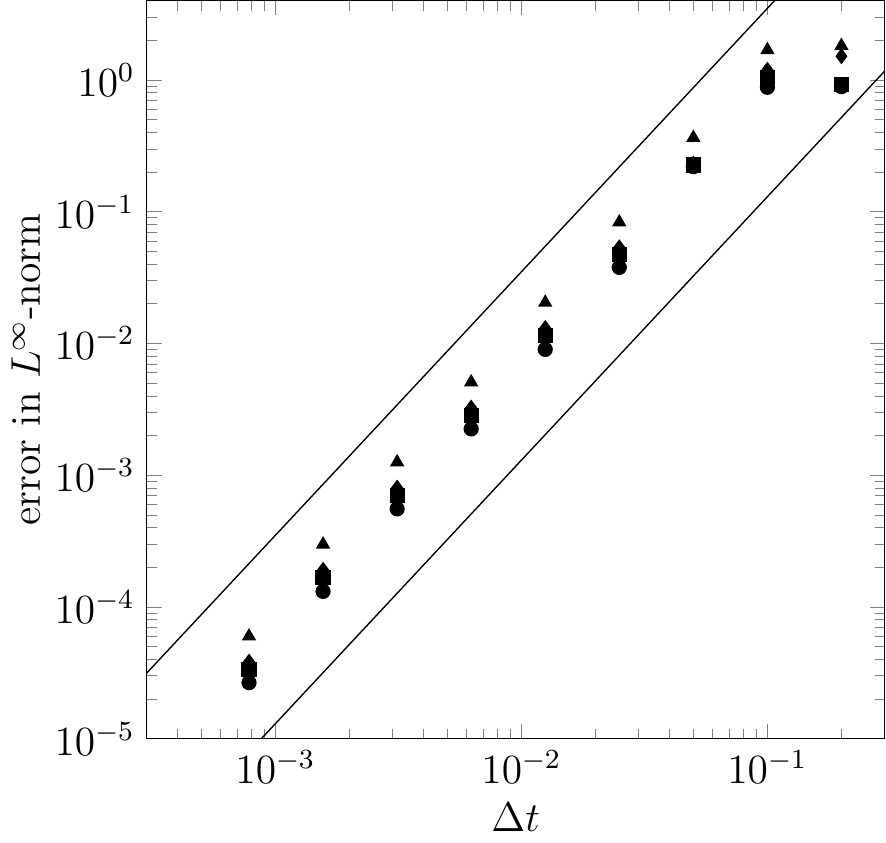}
		\caption{Convergence results for test case 1 in the discrete $L^\infty$-norm measured against the exact solution.}
		\label{fig: test case 1}
	\end{subfigure}
	\hfill
	\begin{subfigure}{0.45\textwidth}
		\includegraphics[width=\textwidth]{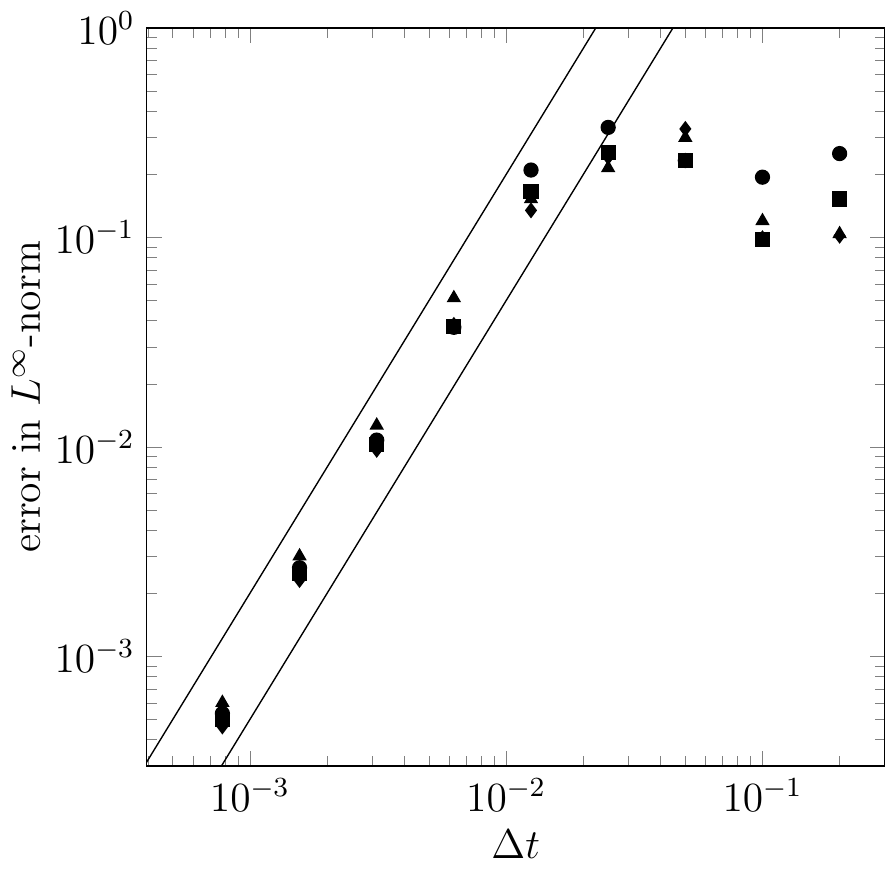}
		\caption{Convergence results for test case 2 in the discrete $L^\infty$-norm measured against the exact solution.}
		\label{fig: test case 2}
	\end{subfigure}
	\\
	\begin{subfigure}{0.45\textwidth}
		\includegraphics[width=\textwidth]{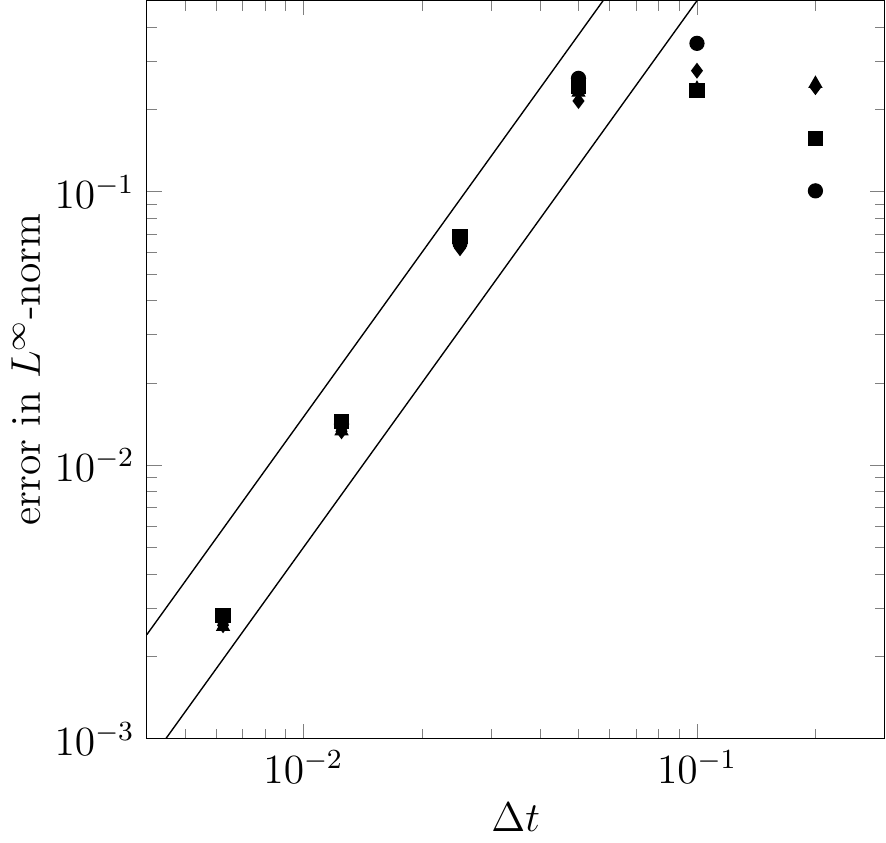}
		\caption{Convergence results for test case 3 in the discrete $L^\infty$-norm measured against a highly resolved numerical solution.}
		\label{fig: test case 3}
	\end{subfigure}
	\caption{Convergence results for test cases 1, 2, and 3 for the isothermal linearized Euler equations, i.e. for $\tau = \frac{1}{2}$, measured in the discrete $L^\infty$-norm. Legend: circles: isothermal flow with background velocity $\vec{u}_0 = \frac{1}{\sqrt{3}} \left( \frac{3}{20}, \frac{1}{10}, \frac{1}{5} \right)^T$; squares: flow with background velocity $\vec{u}_0 = \frac{1}{\sqrt{3}} \left( -\frac{1}{4}, \frac{1}{4}, \frac{1}{2} \right)^T$; diamonds: flow with background velocity $\vec{u}_0 = \frac{1}{\sqrt{3}} \left( \frac{2}{5}, \frac{9}{10}, \frac{3}{4} \right)^T$; triangles: flow with background velocity $\vec{u}_0 = \frac{1}{\sqrt{3}} \left( \frac{3}{4}, \frac{5}{8}, 1 \right)^T$; and solid lines: $2^{\text{nd}}$ order.}
	\label{fig: results}
\end{figure}

\section{Conclusion}

We introduced a structured approach for the derivation of stable LBMs with linear collision operators in sections \ref{sec: relative} and \ref{sec: Partially}.
Instead of deriving collision operators starting from a physical simplifying concept, the partially relative schemes are built on the notion of stability structures as guiding principle.
In section \ref{sec: 3D LEE}, we presented the derivation of partially relative schemes for the example of the 3D isothermal linearized Euler equations.
Due to their linearized structure, these schemes include a multitude of parameters which renders the derivation of stable collision operators challenging if no structured approach is chosen.
Finally, in section \ref{sec: results}, we verified that the collision operator derived in section \ref{sec: 3D LEE} are second-order convergent.

\bibliographystyle{plain}
\bibliography{LBM}

\begin{thebibliography}{10}

\bibitem{banda2006}
Mapundi~K Banda, Wen-An Yong, and Axel Klar.
\newblock A stability notion for lattice boltzmann equations.
\newblock {\em SIAM Journal on Scientific Computing}, 27(6):2098--2111, 2006.

\bibitem{Bardos2000}
Claude Bardos, Fran\c{c}ois Golse, and C.~David Levermore.
\newblock The acoustic limit for the boltzmann equation.
\newblock {\em Archive for Rational Mechanics and Analysis}, 153(3):177--204,
  June 2000.

\bibitem{EntropicLatticeBoltzmannMethods}
Bruce~M Boghosian, Jeffrey Yepez, Peter~V Coveney, and Alexander Wagner.
\newblock Entropic lattice boltzmann methods.
\newblock {\em Proceedings of the Royal Society of London. Series A:
  Mathematical, Physical and Engineering Sciences}, 457(2007):717--766, 2001.

\bibitem{AnalysisTechniques2009}
Alfonso Caiazzo, Michael Junk, and Martin Rheinl{\"a}nder.
\newblock Comparison of analysis techniques for the lattice boltzmann method.
\newblock {\em Computers \& Mathematics with Applications}, 58(5):883 -- 897,
  2009.
\newblock Mesoscopic Methods in Engineering and Science.

\bibitem{ChaiHeGuoShi2018}
Zhenhua Chai, Nanzhong He, Zhaoli Guo, and Baochang Shi.
\newblock Lattice boltzmann model for high-order nonlinear partial differential
  equations.
\newblock {\em Phys. Rev. E}, 97:013304, Jan 2018.

\bibitem{EntropicLBM3D}
SS~Chikatamarla, S~Ansumali, and IV~Karlin.
\newblock Entropic lattice boltzmann models for hydrodynamics in three
  dimensions.
\newblock {\em Physical review letters}, 97(1):010201, 2006.

\bibitem{RelativeVelocity}
Fran{\c{c}}ois Dubois, Tony F{\'e}vrier, and Benjamin Graille.
\newblock Stability of a bidimensional relative velocity lattice boltzmann
  scheme.
\newblock {\em arXiv preprint arXiv:1506.02381}, 2015.

\bibitem{DuboisLallemand2009}
Fran{\c c}ois Dubois and Pierre Lallemand.
\newblock Towards higher order lattice boltzmann schemes.
\newblock {\em Journal of Statistical Mechanics: Theory and Experiment},
  2009(06):P06006, 2009.

\bibitem{DuboisLallemand2011}
Fran{\c c}ois Dubois and Pierre Lallemand.
\newblock Quartic parameters for acoustic applications of lattice boltzmann
  scheme.
\newblock {\em Computers \& Mathematics with Applications}, 61(12):3404--3416,
  6 2011.

\bibitem{GeierSchoenherr2017}
Martin Geier and Martin Sch\``onherr.
\newblock Esoteric twist: An efficient in-place streaming algorithmus for the
  lattice boltzmann method on massively parallel hardware.
\newblock {\em Computation}, 5(2), 2017.

\bibitem{HeLuo97From}
Xiaoyi He and Li-Shi Luo.
\newblock Theory of the lattice boltzmann method: From the boltzmann equation
  to the lattice boltzmann equation.
\newblock {\em Physical Review E}, 56(6):6811, 1997.

\bibitem{Junk2005}
Michael Junk, Axel Klar, and Li-Shi Luo.
\newblock {Asymptotic analysis of the lattice Boltzmann equation}.
\newblock {\em Journal of Computational Physics}, 210(2):676--704, December
  2005.

\bibitem{JunkYong2009}
Michael Junk and Wen-An Yong.
\newblock Weighted {$L^{2}$}-stability of the lattice boltzmann method.
\newblock {\em SIAM Journal on Numerical Analysis}, 47(3):1651--1665, 2009.

\bibitem{kataoka2004lattice}
Takeshi Kataoka and Michihisa Tsutahara.
\newblock Lattice boltzmann method for the compressible euler equations.
\newblock {\em Physical review E}, 69(5):056702, 2004.

\bibitem{OtteDiss}
Philipp Otte.
\newblock {\em {A}nalysis, numerics, and implementation of
  {K}inetic-{C}ontinuum coupling using {L}attice-{B}oltzmann methods}.
\newblock Dissertation, RWTH Aachen University, Aachen, 2018.
\newblock Veröffentlicht auf dem Publikationsserver der RWTH Aachen University
  2019; Dissertation, RWTH Aachen University, 2018.

\bibitem{Otte2016}
Philipp Otte and Martin Frank.
\newblock Derivation and analysis of lattice boltzmann schemes for the
  linearized euler equations.
\newblock {\em Computers \& Mathematics with Applications}, 72(2):311 -- 327,
  2016.
\newblock The Proceedings of \{ICMMES\} 2014.

\bibitem{Rheinlaender2010}
Martin Rheinl\"ander.
\newblock On the stability structure for lattice boltzmann schemes.
\newblock {\em Computers \& Mathematics with Applications}, 59(7):2150 -- 2167,
  2010.
\newblock Mesoscopic Methods in Engineering and ScienceInternational
  Conferences on Mesoscopic Methods in Engineering and Science.

\bibitem{Yong2001}
Wen-An Yong.
\newblock Basic aspects of hyperbolic relaxation systems.
\newblock In Heinrich Freist{\"u}hler and Anders Szepessy, editors, {\em
  Advances in the Theory of Shock Waves}, volume~47 of {\em Progress in
  Nonlinear Differential Equations and Their Applications}, pages 259--305.
  Birkh{\"a}user Boston, 2001.

\bibitem{Yong2009862}
Wen-An Yong.
\newblock An onsager-like relation for the lattice boltzmann method.
\newblock {\em Computers \& Mathematics with Applications}, 58(5):862 -- 866,
  2009.

\bibitem{YongZhaoLuo21016}
Wen-An Yong, Weifeng Zhao, and Li-Shi Luo.
\newblock Theory of the lattice boltzmann method: Derivation of macroscopic
  equations via the maxwell iteration.
\newblock {\em Phys. Rev. E}, 93:033310, Mar 2016.

\end{thebibliography}

\end{document}